\PassOptionsToPackage{lowtilde}{url}
\documentclass[11pt]{amsart}
\usepackage{microtype}               

\usepackage[english]{babel}
\usepackage{amsfonts}
\usepackage{amsmath} 
\usepackage{amssymb}
\usepackage{amsthm}
\usepackage{mathtools}
\usepackage{mathrsfs}				    
\usepackage{verbatim}				    
\usepackage[all]{xy}    	          
\SelectTips{cm}{11}                  
\usepackage[T1]{fontenc}        
\usepackage{graphicx}

\usepackage{enumitem}
\setlist[enumerate]{label=(\roman*)}  

\usepackage{tikz}									
\usetikzlibrary{matrix}
\usetikzlibrary{patterns}

\usepackage{hyperref}
\hypersetup{
  colorlinks   = true,          
  urlcolor     = blue,          
  linkcolor    = blue,          
  citecolor   = red             
}
\newcommand{\N}{{\mathbb{N}}}

\newcommand{\Q}{{\mathbb{Q}}}
\newcommand{\R}{\mathbb{R}}			
\newcommand{\C}{\mathbb{C}}

\newcommand{\calA}{{\mathcal A}}
\newcommand{\calB}{{\mathcal B}}
\newcommand{\calC}{{\mathcal C}}

\newcommand{\calF}{{\mathcal F}}

\newcommand{\calI}{{\mathcal I}}
\newcommand{\calJ}{{\mathcal J}}

\newcommand{\calL}{{\mathcal L}}
\newcommand{\calM}{{\mathcal M}}

\newcommand{\calO}{{\mathcal O}}

\newcommand{\calU}{{\mathcal U}}
\newcommand{\calV}{{\mathcal V}}
\newcommand{\calW}{{\mathcal W}}

\newcommand{\calY}{{\mathcal Y}}

\newcommand{\frakM}{{\mathfrak M}}

\newcommand{\resx}{\mathscr H(x)}		
\newcommand{\rescompl}[1]{\mathscr H\left(#1\right)}	
\newcommand{\m}[1]{\mathcal M\left(#1\right)}		
\newcommand{\X}{\mathscr X}			
\newcommand{\Y}{\mathscr Y}			

\newcommand{\OO}{\mathcal O}		%
\newcommand{\OOO}{\mathcal O^\circ}		%
\newcommand{\OOOO}{\mathcal O^{\circ\circ}}		%
\newcommand{\T}{{T_\X}}		%
\newcommand{\Txz}{T_{X,Z}}		%
\newcommand{\Xeta}{{\X^*}}		%
\newcommand{\I}{\mathbb R_{>0}}		%
	
\newcommand{\boundaryZ}{\partial^{\mathrm{an}}\Txz}		%

\newcommand{\defi}[1]{\emph{#1}}

\DeclareMathOperator{\Spf}{Spf}
\DeclareMathOperator{\Spec}{Spec}

\DeclareMathOperator{\tors}{tors}
\DeclareMathOperator{\Div}{Div}
\DeclareMathOperator{\forg}{for}

\DeclareMathOperator{\Hom}{Hom}

\DeclareMathOperator{\Sp}{sp}		
\DeclareMathOperator{\Bl}{Bl}		

\DeclareMathOperator{\trdeg}{trdeg}
\DeclareMathOperator{\rank}{rank}

\newlength{\mylength}
\newtheoremstyle{plain2}    
  {}            
  {}            
  {\itshape}    
  {}            
  {\bfseries}   
  {.}           
  {5pt plus 1pt minus 1pt}  
  {{\thmnumber{(#2)} \thmname{#1}{\thmnote{ (#3)}}}}          
\theoremstyle{plain2}

\newcounter{thm}[section] 
\newtheorem{thm}[subsection]{Theorem}
\newtheorem{cor}[subsection]{Corollary}
\newtheorem{lem}[subsection]{Lemma}
\newtheorem{prop}[subsection]{Proposition}

\newtheoremstyle{definition2}    
  {}   
  {}   
  {\normalfont}  
  {}       
  {\bfseries} 
  {.}        
  {5pt plus 1pt minus 1pt} 
  {{(\thmnumber{#2}) \thmname{#1}{\thmnote{ (#3)}}}}          
\theoremstyle{definition2}

\newtheorem{rem}[subsection]{Remark}
\newtheorem{ex}[subsection]{Example}

\newcounter{stepcounter}[subsection]     
\newtheoremstyle{stepstyle}
  {2pt}   
  {2pt}   
  {\normalfont}  
  {\parindent}       
  {\itshape} 
  {}         
  {5pt plus 1pt minus 1pt} 
  {{\thmname{#1} \thmnumber{#2}:{\thmnote{ (#3)}}}}          
\theoremstyle{stepstyle}
\newtheorem{step}[stepcounter]{Step}

\newtheoremstyle{point}
  {}   
  {}   
  {\normalfont}  
  {}       
  {\bfseries} 
  {}         
  {5pt plus 1pt minus 1pt} 
  {{\thmname{#1}(\thmnumber{#2})\thmnote{ #3.}}}          
\theoremstyle{point}
\newtheorem{point}[subsection]{}

\newcommand{\pa}[1]{\begin{point}#1\end{point}}              
\newcommand{\Pa}[2]{\begin{point}[#1]#2\end{point}}          

\newtheoremstyle{point*}
  {}   
  {}   
  {\normalfont}  
  {}       
  {\bfseries} 
  {}         
  {5pt plus 1pt minus 1pt} 
  {{\thmname{#1}\thmnote{ #3.}}}          
\theoremstyle{point*}
\newtheorem{point*}[subsubsection]{}

\numberwithin{equation}{subsection}

\newtheoremstyle{subpoint}
  {}   
  {}            
  {\normalfont}  
  {}                   
  {\normalfont} 
  {}         
  {5pt plus 1pt minus 1pt} 
  {{\thmname{#1}(\thmnumber{#2})\thmnote{ #3.}}}          
\theoremstyle{subpoint}
\newtheorem{subpoint}[equation]{}

\makeatletter
\renewenvironment{proof}[1][\proofname]{\par
  \vspace{-\topsep}
  \pushQED{\qed}%
  \normalfont
  \topsep0pt \partopsep0pt 
  \trivlist
  \item[\hskip\labelsep
        \itshape
    #1\@addpunct{.}]\ignorespaces
}{%
  \popQED\endtrivlist\@endpefalse
  \addvspace{6pt plus 6pt} 
}
\makeatother

\usepackage{color}

\title{Normalized Berkovich spaces and surface singularities}
\author{Lorenzo Fantini}
\date{\today}
\address{Institut Math\'ematique de Jussieu, Universit\'e Pierre et Marie Curie, Paris, France}
\email{\href{mailto:lorenzo.fantini@imj-prg.fr}{lorenzo.fantini@imj-prg.fr}}
\urladdr{\url{https://webusers.imj-prg.fr/~lorenzo.fantini}}%

\begin{document}

\begin{abstract}
We define normalized versions of Berkovich spaces over a trivially valued field $k$, obtained as quotients by the action of $\I$ defined by rescaling semivaluations.
We associate such a normalized space to any special formal $k$-scheme and prove an analogue of Raynaud's theorem, characterizing categorically the spaces obtained in this way.
This construction yields a locally ringed $G$-topological space, which we prove to be $G$-locally isomorphic to a Berkovich space over the field $k((t))$ with a $t$-adic valuation. 
These spaces can be interpreted as non-archimedean models for the links of the singularities of $k$-varieties, and allow to study the birational geometry of $k$-varieties using techniques of non-archimedean geometry available only when working over a field with non-trivial valuation.
In particular, we prove that the structure of the normalized non-archimedean links of surface singularities over an algebraically closed field $k$ is analogous to the structure of non-archimedean analytic curves over $k((t))$, and deduce characterizations of the essential and of the log essential valuations, i.e. those valuations whose center on every resolution (respectively log resolution) of the given surface is a divisor.
\end{abstract}

\maketitle

\section{Introduction}

%
%

Berkovich's geometry is an approach to non-archimedean analytic geometry developed in the late nineteen-eighties and early nineteen-nineties by Berkovich in \cite{Ber90} and \cite{Ber93}.
To overcome the problems given by the fact that the metric topology of any valued field is totally disconnected, Berkovich adds many points to the usual points of a variety $X$ (not unlike what happens in algebraic geometry with generic points), to obtain an \emph{analytic space} $X^{\mathrm{an}}$, which is a locally ringed space with very nice topological properties and whose points can be seen as \emph{real semivaluations}.

One important feature of Berkovich's theory is that it works also over a trivially valued base field, for example $\C$. 
This gives rise to objects that are far from being trivial, resembling some spaces studied in valuation theory, but carrying in addition an analytic structure, and containing a lot of information about the singularities of $X$.
For example, Thuillier \cite{Thu07} obtained the following result (generalizing a theorem by Stepanov): if $X$ is a variety over a perfect field $k$, then the homotopy type of the dual intersection complex of the exceptional divisor of a log resolution of $X$ does not depend on the choice of the log resolution. 
To prove this, he associates to a subvariety $Z$ of a $k$-variety $X$ a $k$-analytic space that can be called the \emph{punctured tubular neighborhood} of $Z^{\mathrm{an}}$ in $X^{\mathrm{an}}$.
It is a subspace of $X^{\mathrm{an}}$, invariant under modifications of the pair $(X,Z)$, consisting of all the semivaluations on $X$ that have center on $Z$ but are not semivaluations on $Z$. 

In this paper we define a normalized version $\Txz$ of this punctured tubular neighborhood, by taking the quotient of the latter by the group action of $\I$ that corresponds to rescaling semivaluations. 
The space $\Txz$ can be thought of as a \emph{non-archimedean model} of the \emph{link} of $Z$ in $X$. 
It is a locally ringed space in $k$-algebras, endowed with the pushforward of the Grothendieck topology and structure sheaf from the punctured tubular neighborhood, and it can be seen as a wide generalization of Favre and Jonsson's \emph{valuative tree}, an object that has important applications to the dynamics of complex polynomials in two variables.
Indeed, the valuative tree is homeomorphic to the topological space underlying $T_{\mathbb A^2_\C,\{0\}}$, but the latter has much more structure. 
Moreover, $\Txz$ can also be thought of as a compactification of the normalized valuation space considered in \cite{JonssonMustata12} and \cite{BoucksomdeFernexFavreUrbinati13}, as the latter is homeomorphic to the subset of $\Txz$ consisting of all the points that are actual valuations on the function field of $X$.
Valuation spaces homeomorphic to $T_{\mathbb A^n_\C,\{0\}}$ appear also in \cite{BoucksomFavreJonsson08} ; there the authors develop the basics of pluripotential theory on those spaces, building on previous work of Favre and Jonsson in dimension 2.

More generally, we associate a \emph{normalized space} $\T$ to any \emph{special formal $k$-scheme} $\X$.
If $X$ is a $k$-variety and $Z$ is a closed subvariety of $X$, then the formal completion $\widehat{X/Z}$ of $X$ along $Z$ is a special formal $k$-scheme, and we have $T_{\widehat{X/Z}}\cong\Txz$.

The crucial property of $\T$ is the following: while not an analytic space itself, as a locally ringed $G$-topological space in $k$-algebras the normalized space $\T$ is $G$-locally isomorphic to an analytic space over the field $k((t))$ with a $t$-adic absolute value.
Attention should be paid to the fact that these local isomorphisms are not canonical, and in general they do not induce a global $k((t))$-analytic structure on $\T$.
In particular, this result explains why the valuative tree looks so much like a Berkovich curve defined over $\C((t))$.
This interpretation permits to study $\T$, and thus deduce information about $\X$, with various tools of non-archimedean analytic geometry, including the ones that work only over non-trivially valued fields.
We have only recently learned about the article \cite{Ben-BassatTemkin2013}. There the authors use the punctured tubular neighborhood of $Z^{\mathrm{an}}$ in $X^{\mathrm{an}}$ to encode the descent data necessary to glue a coherent sheaf on a formal neighborhood of $Z$ to a coherent sheaf on $X\setminus Z$.
The result we just described on the structure of $\T$ is conceptually similar to the results of Sections 4.2 to 4.6 of {\it loc. cit.}.

We define an affinoid domain of $\T$ as a $G$-admissible subspace $V$ of $\T$ that is isomorphic to a strict $k((t))$-affinoid space, and we show that this definition does not depend on the choice of a $k((t))$-analytic structure on $V$.
This is done by showing, following \cite{Liu90}, that a reduced $k((t))$-analytic space is strictly affinoid if and only if it is Stein, compact and its ring of analytic functions bounded by one is a special $k$-algebra.
Every normalized space $\T$ is compact and $G$-covered by finitely many affinoid domains, and this allows us to characterize the category of all the locally ringed $G$-topological spaces of the form $\T$ as the localization of the category of special formal $k$-schemes by the class of admissible formal blowups.
This is a ``normalized spaces version'' of a classical theorem of Raynaud for non-archimedean analytic spaces (see \cite[4.1]{BosLut85}).

We then apply normalized spaces to the study of surface singularities.
While the importance of valuations in the study of resolutions of surface singularities was emphasized already in the work of Zariski and Abhyankar (see \cite{Zariski39} and \cite{Abhyankar56}), in our work also the additional structure given by the sheaf of analytic functions plays an important role.
If $k$ is an algebraically closed field, $X$ is a $k$-surface and $Z$ is a subspace of $X$ containing its singular locus, by the structure theorem discussed above the normalized space $\Txz$ behaves like a non-archimedean analytic curve over $k((t))$. 
The theory of such curves is well understood, thanks to work of Bosch and L\"utkebohmert \cite{BosLut85} (\cite[Chapter 4]{Ber90} for Berkovich spaces).
In particular, there is a correspondence between (semistable) models and (semistable) vertex sets (see \cite{Duc}, \cite[Chapter 6]{Temkin15} and \cite{BakerPayneRabinoff14}).
We prove an analogue of this result for the normalized space $\Txz$.
After showing how to construct formal log modifications of the pair $(X,Z)$ with prescribed exceptional divisors, we characterize among those modifications the ones that correspond to a log resolution of $(X,Z)$ by performing a careful study (analogous to \cite[2.2 and 2.3]{BosLut85} and \cite[4.3.1]{Ber90}) of the fibers of the map sending every semivaluation to its center on the modification.
Our main source of inspiration in developing this approach was Ducros's work \cite{Duc}. 

The strategy described above leads to two characterizations in terms of the local structure of $\Txz$ of the \defi{essential} and \defi{log essential valuations} on $(X,Z)$, i.e. those valuations whose center on every resolution (respectively log resolution) of $(X,Z)$ is a divisor.
Whenever $k$ is the field of complex numbers and $Z$ is the singular locus of $X$, this is related to a famous conjecture of Nash from the nineteen-seventies (but published only in 1995 in \cite{Nash95}) involving the \emph{arc space} $X_\infty$ of a complex variety $X$.
Nash constructed an injective map from the set of irreducible components of the subspace of $X_\infty$ consisting of the arcs centered in the singular locus of $X$ to the set of essential valuations on $X$, and asked whether this map is surjective.
While this is known to be false if $\mathrm{dim}(X)\geq3$ (see \cite{deFernex}), for complex surfaces a proof was given by Fern\'andez de Bobadilla and Pe Pereira in \cite{deBobadillaPereira12}.
More recently de Fernex and Docampo \cite{deFernexDocampo14} proved that in arbitrary dimension every valuation that is terminal with respect to the minimal model program over $X$ is in the image of the Nash map, deducing a new proof of de Bobadilla--Pereira's theorem.
The class of log essential valuations can be larger than the set of Nash's essential valuations, since in some cases the exceptional locus of the minimal resolution of $X$ may not be a divisor with normal crossings.
However, for many classes of singularities (e.g. rational singularities) these two notions coincide.


We now give a short overview of the content of the paper. 
In Section~\ref{section_preliminaries} we recall some basic constructions of the theories of formal schemes and Berkovich spaces.
In Section~\ref{section_3.1} we define the normalized space of a special formal $k$-scheme, while in Section~\ref{section_3.2} we prove the structure theorem of normalized spaces and deduce some interesting consequences.
In Section~\ref{section_3.3} we define affinoid domains in a normalized space, and show that this notion is independent of the choice of a $k((t))$-analytic structure.
This leads to the normalized version of Raynaud's theorem in Section~\ref{section_3.4}.
We then move to the study of pairs $(X,Z)$, where $X$ is a $k$-surface and $Z$ is a closed subvariety of $X$ containing its singular locus.
Section~\ref{section_4.1} contains the correspondence theorem between formal modifications of $(X,Z)$ and vertex sets.
In Section~\ref{section_4.2} we study discs and annuli in $\Txz$; they are used in Section~\ref{section_4.3}, where we describe the formal fibers of the specialization map.
In Section~\ref{section_4.4} we show how these techniques lead to the characterization of log essential and essential valuations.

Several examples have been given for the reader who might want to quickly reach a basic understanding of the applications of normalized spaces to the study of surface singularities, without spending much time learning formal and Berkovich geometry.
This reader should pay attention to the examples \ref{analytification_functor}, \ref{R: interpretation algebraic case}, \ref{example_explicit_normalization}, \ref{analytic_structure_valuative_tree}, and might benefit from reading the short note \cite{Fantini14}, where some of the results of this paper were announced.

\subsection*{Acknowledgments} 
The results of this paper were part of my PhD thesis at KU Leuven.
I am very thankful to my advisor, Johannes Nicaise.
I am also grateful to Nero Budur, Antoine Ducros, Charles Favre, Mircea Musta{\c{t}}{\u{a}}, Sam Payne, C\'edric P\'epin, Michael Temkin, Amaury Thuillier, and Wim Veys, for helpful discussions and comments.
I am extremely grateful to an anonymous referee for his/her thorough reading of the manuscript and numerous helpful comments.
I acknowledge the support of the Fund for Scientific Research - Flanders (grant G.0415.10) and of the European Research Council (Starting Grant project ``Nonarcomp'' no.307856).


\vspace{2pt}\section{Special formal schemes and their Berkovich spaces}\label{section_preliminaries}

In the section we recall the notions of special formal schemes and the associated Berkovich spaces.
For a detailed study of noetherian formal schemes we refer the reader to \cite{Illusie} or \cite{Bosch14}; a quick introduction is \cite{Nicaise08}.
Special formal schemes are treated for example in \cite{deJ95} and \cite{Ber96}.

\pa{
Let $R$ be a complete discrete valuation ring, $K$ the fraction field of $R$ and $k$ its residue field. 
By definition we allow $R$ to be a trivially valued field $k$.
A \defi{formal $R$-scheme} is a noetherian formal scheme endowed with a (not necessarily adic) morphism of noetherian formal schemes to $\Spf R$.
Recall that a morphism of noetherian formal schemes $f:\Y\to\X$ is said to be \defi{adic} if $f^*(\calJ)\calO_\Y$ is an ideal of definition of $\Y$ for some (and thus for all) ideal of definition $\calJ$ of $\X$.
}

\pa{
A topological $R$-algebra $A$ is said to be a \defi{special $R$-algebra} if it is a noetherian adic ring and the quotient $A/J$ is a finitely generated $R$-algebra for some ideal of definition $J$ of $A$.
A formal $R$-scheme $\X$ is said to be a \defi{special formal $R$-scheme} if it is separated and locally isomorphic to the formal spectrum of a special $R$-algebra.
In particular the reduction $\X_0$ of $\X$ is a reduced and separated scheme locally of finite type over $k$.
Observe that $\X_0$ is generally different the special fiber of $\X$, which is by definition the special formal $k$-scheme $\X_s=\X\times_R k$.
}

\pa{
By \cite[1.2]{Ber96}, special $R$-algebras are exactly the adic $R$-algebras of the form 
\[
R\{X_1,\ldots,X_n\}\lbrack\lbrack Y_1,\ldots,Y_m \rbrack\rbrack/I\cong R[[Y_1,\ldots,Y_m]]\{X_1,\ldots,X_n\}/I,
\]
with ideal of definition generated by an ideal of definition of $R$ and by the $Y_i$'s. 
Recall that if $A$ is a $I$-adic topological ring, then $A\{X_1,\ldots,X_n\} \coloneqq \varprojlim_{\ell\geq1}\big(A/I^\ell\big)[X_1,\ldots,X_n]$ is the algebra of convergent power series over $A$ in the variables $(X_1,\ldots,X_n)$.
Since every $R$-algebra topologically of finite type is special (we can take $m=0$ above), every formal $R$-scheme of finite type is a special formal $R$-scheme.
On the other hand, a special formal $R$-scheme is of finite type if and only if its structure morphism to $\Spf(R)$ is adic.
}

\begin{ex}
When working over a trivially valued field $k$, we have an isomorphism of $k$-algebras $k\{X\}\lbrack\lbrack Y \rbrack\rbrack \cong k \lbrack X\rbrack \lbrack\lbrack Y \rbrack\rbrack$.
The latter is not isomorphic to $k[[Y]][X]$ but to the bigger $k$-algebra $k[[Y]]\{X\}$, which consists of the $Y$-adically convergent power series in $X$ with coefficients in $k[[Y]]$.
\end{ex}

\Pa{Example: the algebraic case}{\label{example_algebraic}
If $X$ is a separated $R$-scheme locally of finite type and $Z$ is a subscheme of the special fiber $X\otimes_R k$ of $X$, then the formal completion $\X=\widehat{X/Z}$ of $X$ along $Z$ is a special formal $R$-scheme. 
In this case, $\X_0=Z_{\mathrm{red}}$. 
For example, if $X= \mathbb A^2_R =\Spec\big(R[X_1,X_2]\big)$ and $Z$ is the origin of the special fiber of $X$, then $\widehat{X/Z}\cong\Spf\big(R\lbrack\lbrack X_1,X_2\rbrack\rbrack\big)$; the special fiber of $\widehat{X/Z}$ is $\Spf k[[X_1,X_2]]$ and its reduction is $\Spec k$.
Similarly, the formal completion of a special formal $k$-scheme along a closed subscheme of its special fiber is again a special formal $k$-scheme.
}

\pa{
All special $R$-algebras are excellent.
This follows from \cite[7]{Valabrega75} when the characteristic of $K$ is positive and from \cite[9]{Valabrega76} when it is zero.
A special formal $R$-scheme $\X$ is said to be \defi{normal} if it can be covered by affine subschemes $\Spf(A)$ with $A$ normal. Since the rings $A$ are excellent, this is equivalent to the normality of all completed local rings of $\X$. 
}

\pa{
If $A$ is a special $k$-algebra and $T$ is an element of an ideal of definition of $A$, then $T$ is topologically nilpotent and therefore it induces a homomorphism $k[[t]]\to A$ that canonically makes $A$ into a special $k[[t]]$-algebra.
Conversely, any special $k[[t]]$-algebra is canonically a special $k$-algebra.
We will sometimes denote a special formal $k[[t]]$-scheme by $\X_t$; and $\X$ will then denote $\X_t$ seen as a special formal $k$-scheme.
}

\pa{
Let $\X$ be a noetherian formal scheme with largest ideal of definition $\calJ$ and let $\calI$ be a coherent ideal sheaf on $\X$. 
The \defi{formal blowup of $\X$ along $\calI$} is the $\X$-formal scheme
$$
\X':=\lim_{\stackrel{\longrightarrow}{n\geq
1}}\mathrm{Proj}\left(\oplus_{d=0}^{\infty}\mathcal{I}^d\otimes_{\mathcal{O}_{\X}}(\mathcal{O}_{\X}/\mathcal{J}^n)\right).
$$
We call the closed formal subscheme of $\X$ defined by $\calI$ the \defi{center} of the blowup. 
The formal blowup $\X'\to\X$ of $\X$ along $\calI$ is characterized by the following universal property (see \cite[8.2.9]{Bosch14}): $\X'$ is a noetherian formal scheme such that the ideal $f^{-1}\calI\OO_{\X'}$ is invertible on $\X'$, and every morphism of noetherian formal schemes $\Y\to\X$ such that $f^{-1}\calI\OO_{\Y}$ is invertible on $\Y$ factors uniquely through a morphism of noetherian formal schemes $\Y\to\X'$.
We say that the blowup $\X'\to\X$ of $\X$ along $\calI$ is \defi{admissible} if the ideal $\calI$ is $\calJ$-open, i.e. contains a power of $\calJ$. 
}

\begin{ex}
Let $X$ be a noetherian scheme, let $Z$ be a closed subscheme of $X$ defined by a coherent ideal sheaf $\calJ$ and denote by $\X=\widehat{X/Z}$ the formal completion of $X$ along $Z$. 
Let $\calI$ be a $\calJ$-open coherent ideal sheaf on $X$, and $\hat\calI$ the induced ideal sheaf on $\X$.
Then the formal blowup of $\X$ along $\hat\calI$ is isomorphic to the formal completion along $f^*\calJ\OO_{\Bl_\calI(X)}$ of the blowup $\Bl_\calI(X)$ of $X$ along $\calI$, where $f:\Bl_\calI(X)\to X$ is the blowup.
This is \cite[2.16.(5)]{Nicaise09}.
\end{ex}

\pa{\label{basic properties blowup} Admissible formal blowups share many properties with blowups of ordinary schemes. In particular, the following facts are proved as for schemes:
\begin{enumerate}[ref=\ref{basic properties blowup}.\roman{enumi}]
\item \label{blowup: composition} a composition of admissible blowups is an admissible blowup (\cite[8.2.11]{Bosch14});
\item \label{blowup: domination} two admissible blowups can be dominated by a third one (\cite[8.2.16]{Bosch14} and the previous point);
\item \label{blowup: extension} an admissible blowup of an open formal subscheme of $\X$ can be extended to an admissible blowup of $\X$ (\cite[8.2.13]{Bosch14}).
\end{enumerate}
}

\pa{
An admissible blowup of a special formal $R$-scheme is a special formal $R$-scheme by~\cite[2.17]{Nicaise09}. Similarly, an admissible blowup of a formal $R$-scheme of finite type is of finite type.
}

\Pa{Berkovich theory}{
Berkovich's approach to non-archimedean analytic geometry was developed in \cite{Ber90} and \cite{Ber93}; a good introduction to the theory is \cite{Temkin15}.
Since the general definition of a $K$-analytic space is quite technical, we will content ourselves with listing some properties of $K$-analytic spaces and introducing via examples those spaces that appear in the rest of the paper.
In particular, we will recall how to associate a $K$-analytic space $\X^\beth$ to a special formal $R$-scheme $\X$ and define the specialization map.
This construction was introduced for rigid spaces by Berthelot in \cite{Berthelot} (see also \cite[\S7]{deJ95} for a detailed exposition), while in the context of Berkovich spaces it was studied in \cite{Berkovich94} and \cite{Ber96}. 
If $\X$ is special over a trivially valued field $k$, we will also study a subspace of $\X^\beth$, introduced by Thuillier in \cite{Thu07}, that behaves more like a generic fiber for $\X$ (see also \cite{Ben-BassatTemkin2013}).
}

\pa{
A \emph{$K$-analytic space} is a locally compact and locally path connected topological space $X$ with the following additional structures:
\begin{enumerate}
\item For every point $x$ of $X$, a completed valued field extension $\resx$ of $K$, called the \defi{completed residue field} of $X$ at $x$.
\item A $G$-topology on $X$, whose $G$-admissible subspaces are called \defi{analytic domains} of $X$.
\item A local $G$-sheaf in $K$-algebras $\calO_X$ on $X$, the \defi{sheaf of analytic functions}.
\end{enumerate}
A $G$-topology is a simple kind of Grothendieck topology, we refer to \cite[\S9.1]{BGR} for the definitions.
The $G$-topology is finer than the usual topology of $X$, i.e. every open subset of $X$ is an analytic domain and every open cover of an analytic domain is a $G$-cover.
If $V$ is an analytic domain of $X$, $x\in V$, $f\in\calO_X(V)$, then $f$ can be evaluated in $x$, yielding an element $f(x)$ of $\resx$.
Therefore, also $|f(x)|\in\R_+$ makes sense.
We refer to \cite{Temkin15} for the general definition of the category $(An_K)$.
}

\begin{ex}\label{analytification_functor}
A fundamental example of $K$-analytic space is the analytification $X^{\mathrm{an}}$ of a $K$-scheme of finite type $X$.
As a topological space,
\[
X^{\mathrm{an}}=\big\{(\xi_x,|\cdot|_x)\big|\xi_x\in X, |\cdot|_x\text{ abs. value on }\kappa(\xi_x)\text{ extending the one of }K\big\},
\]
with the weakest topology such that the map $\rho\colon X^{\mathrm{an}}\to X$ sending a point $x=(\xi_x,|\cdot|_x)$ to $\xi_x$ is continuous, and for each open $U$ of $X$ and each element $f$ of $\mathcal O_X(U)$ the induced map $\rho^{-1}(U)\to\R$ sending $x$ to $|f(x)|=|f|_x$ is continuous. 
The field $\rescompl{x}$ is the completion of $\kappa(\xi_x)$ with respect to $|\cdot|_x$.
A morphism of $K$-schemes of finite type $Y\to X$ induces a map $Y^{\mathrm{an}}\to X^{\mathrm{an}}$.
The space $X^{\mathrm{an}}$ is connected (respectively Hausdorff, compact) if and only if $X$ is connected (respectively separated, proper).
Moreover, whenever $X$ is proper then GAGA type theorems hold (see \cite[\S3.4, \S3.5]{Ber90}).
The sheaf $\mathcal O_{X^{\mathrm{an}}}$ can be thought of as a completion of the sheaf $\mathcal O_{X}$ with respect to some seminorm.
For example, the analytic functions on an open $U$ of $\mathbb A^{n,\mathrm{an}}_K$ are the maps $f:U\to\coprod_{x\in U}\rescompl{x}$ that are locally uniform limits of rational functions without poles.
More generally, every locally closed subspace of $X^{\mathrm{an}}$ can be canonically given the structure of a reduced $K$-analytic space.
\end{ex}

\pa{
The $G$-topology of a $K$-analytic space $X$ is constructed from an important class of distinguished compact analytic domains of $X$, that of \defi{affinoid domains}.
Recall that an affinoid $K$-algebra is a quotient of a Banach $K$-algebra of the form $K\big\{r_1^{-1}T_1,\ldots,r_n^{-1}T_n\big\}=
\big\{\sum_{\underline i\in\N^n} a_{\underline i}\underline T^{\underline i} \,\big|\, a_{\underline i}\in K, \, \lim_{|\underline i|\to\infty}|a_{\underline i}|\underline r^{\underline i}=0\big\}$ (where $r_i>0$, and the Banach norm is given by $||\sum a_{\underline i}\underline T^{\underline i}||=\max{|a_{\underline i}|\underline r^{\underline i}}$), and that the affinoid spectrum $\m{\calA}$ of $\calA$ is the set of bounded multiplicative seminorms on $\calA$, with the topology of pointwise convergence.
Affinoid domains are then some distinguished subsets of $X$ homeomorpic to the affinoid spectrum $\m{\calA}$ of an affinoid $K$-algebra $\calA$, and an affinoid domain is said to be strict if we can take all $r_i$ above to be equal to 1.
If $V\cong\m{\calA}$ is an affinoid domain of $X$, then $\calO_X(V)\cong\calA$.
A subset $U$ of $X$ is then an analytic domain if and only of for every element $u$ of $U$ there exist finitely many affinoid domains $U_1,\ldots,U_n$ of $X$ contained in $U$ and such that $u\in\cap_iU_i$ and $\cup_iU_i$ is a neighborhood of $u$ in $U$.
In particular, any analytic domain of $X$ is $G$-covered by the affinoid domains it contains.
}

\begin{ex}\label{example_berkovich_affinoid}
The \defi{analytic affine $n$-space} $\mathbb A^{n,\mathrm{an}}_K=\Spec(K[T_1,\ldots,T_n])^{\mathrm{an}}$ can be written as the union of the closed polydiscs $D^n(\underline r)=\{|T_i|\leq r_i\text{ for all }i\}$ of center $0$ and polyradius $r=(r_1,\ldots,r_n)\in(\R_+^*)^n$. The polydisc $D^n(\underline r)$ is an affinoid domain, with associated affinoid $K$-algebra $K\big\{r_1^{-1}T_1,\ldots,r_n^{-1}T_n\big\}$.
Explicit descriptions of the topological spaces underlying $\mathbb A^{1,\mathrm{an}}_K$, $\mathbb A^{1,\mathrm{an}}_k$ and $\mathbb A^{2,\mathrm{an}}_k$ can be found in \cite{Payne15}.
\end{ex}

We will now see how to associate a $K$-analytic space to a special formal $R$-scheme.

\pa{\label{definition_beth_space}
If $\X$ is an affine special formal $R$-scheme of the form
$$\X=\Spf\left(\frac{R\{X_1,\ldots,X_n\}\lbrack\lbrack Y_1,\ldots,Y_m\rbrack\rbrack}{(f_1,\ldots,f_r)}\right),$$
then the associated Berkovich space is
$$\X^\beth = V(f_1,\ldots,f_r)\subset D^n\times_K D^{m}_-\subset\mathbb A^{n+m,\mathrm{an}}_K,$$
where $D^n=D^n(\underline 1)$ is the $n$-dimensional closed unit disc in $\mathbb A^{n,\mathrm{an}}_K$ (as in Example~\ref{example_berkovich_affinoid}), $D^{m}_-=\big\{x\in\mathbb A^{m,\mathrm{an}}_K\;\big|\;|T_i(x)|<1\text{ for all }i\big\}$ is the $m$-dimensional open unit disc in $\mathbb A^{m,\mathrm{an}}_K$ and $V(f_1,\ldots,f_r)$ denotes the zero locus of the $f_i$. 
This construction is functorial, sending an open immersion to an embedding of a closed subdomain, therefore it globalizes to general special formal $R$-schemes by gluing.
If $\X$ is of finite type over $R$, this construction coincides with the one by Raynaud (see \cite{Raynaud74} or \cite[\S4]{BosLut93}) and $\X^\beth$ is compact.
}

\begin{ex}
If $\X=\Spf\big(R\{T\}\big)$, then $\X^\beth$ is the closed unit disc in $\mathbb A^{1,\mathrm{an}}_K$. 
If $\X=\Spf\big(R[[T]]\big)$, then $\X^\beth$ is the open unit disc in $\mathbb A^{1,\mathrm{an}}_K$.
Note that if $K=k$ is trivially valued, the latter is homeomorphic to the interval $\left[0,1\right[$.
\end{ex}

\pa{\label{example_affine_beth_space}
If $\X=\Spf\big(R\{X_1,\ldots,X_n\}\lbrack\lbrack Y_1,\ldots,Y_m\rbrack\rbrack/(f_1,\ldots,f_r)\big)$ is an affine special formal $R$-scheme, then its associated Berkovich space $\X^\beth$ is the increasing union $\X^\beth=\bigcup_{0<\varepsilon<1}W_\varepsilon$, where $W_\varepsilon$ is the subspace of $\X^\beth$ cut out by $|Y_i|\leq1-\varepsilon$.
Moreover, $W_\varepsilon$ is an affinoid domain of $\X^\beth$, with associated affinoid \mbox{$K$-algebra} ${K\big\{X_1,\ldots,X_n,(1-\varepsilon)^{-1}Y_1,\ldots,(1-\varepsilon)^{-1}Y_m\big\}}/{(f_1,\ldots,f_r)}.$
}

\pa{\label{canonical_injection_special_algebras}
Let $A$ be a special $R$-algebra and set $X=(\Spf A)^\beth$.
Then the canonical homomorphism $A \otimes_R K \to \calO_X(X)$ is injective.
Indeed, let $f$ be an element of $A \otimes_R K$ which vanishes in $\calO_X(X)$, and let $\frakM$ be a maximal ideal 
of $A \otimes_R K$.
By \cite[Lemma~7.1.9]{deJ95} $\frakM$ corresponds to a point $x$ of $X$, and the image $f(x)$ of $f$ in the completed local ring of $X$ at $x$ coincides with the 
image $\alpha(f_\frakM)$ via the completion morphism $\alpha\colon (A \otimes_R K)_\frakM \to (A \otimes_R K)_\frakM^{\phantom{\frakM}\wedge}$ of the image $f_\frakM$ of $f$ in the localization of $A \otimes_R K$ at $\frakM$.
% ...
It follows that $\alpha(f_\frakM)=0$, hence $f_\frakM=0$ because $(A \otimes_R K)_\frakM$, being a localization of the commutative noetherian ring $A$, is a local noetherian ring.
Since this is true for every maximal ideal of $A \otimes_R K$, it follows that $f=0$.
}

\pa{\label{definition_specialization_map}
There is a natural \defi{specialization map} $\Sp_\X:\X^\beth\longrightarrow\X_0$ that is defined as follows.
If $\X=\Spf(A)$ is affine, a point of $\X^\beth$ gives rise to a continuous character $\chi_x\colon A\to\rescompl{x}^\circ$, where we have denoted by $\rescompl{x}^\circ$ the valuation ring of $\rescompl{x}$, which in turn gives rise to a character $\widetilde\chi_x\colon A/{I}\to\widetilde{\rescompl{x}}$, where $I$ is the largest ideal of definition of $A$.
The kernel of $\widetilde\chi_x$ is by definition the point $\Sp_\X(x)\in\X_0=\Spec(A/I)$.
If $\calU$ is an open formal subscheme of $\X$ then $\calU^\beth\cong\Sp_\X^{-1}(\calU_0)$, and the restriction of $\Sp_\X$ to the latter coincides with $\Sp_\calU$.
Therefore, the definition we gave extends to general special formal $R$-schemes.
The map $\Sp_\X$ is anticontinuous, i.e. the inverse image of an open subset of $\X$ is closed.
For example, if $\X=\Spf(A)$ and $Z$ is a closed subset of $\X$ defined by an ideal $(f_1,\ldots,f_r)$, then 
\[
\Sp_\X^{-1}(Z)=\big\{ x\in\X^\beth \,\big|\, |f_i(x)|<1 \text{ for all }i=1,\ldots,r \big\}.
\]
As in \cite[0.2.6]{Berthelot}, $\Sp_\X$ can be viewed as a morphism of locally ringed sites $\Sp_\X:\X^\beth\to\X$.
Note that this map is often called also \defi{reduction map}.
}

\pa{\label{lemma fiber reduction}
If $f:\Y\to\X$ is a morphism of special formal $k$-schemes, then 
$\Sp_{\X}\circ f^\beth=f\circ\Sp_\Y$.
Moreover, if $Z$ is a subscheme of $\X_0$, then by \cite[0.2.7]{Berthelot} (or \cite[1.3]{Ber96}) the canonical morphism of formal $k$-schemes $\widehat{\X/Z}\to\X$ induces an isomorphism of $k$-analytic spaces $(\widehat{\X/Z})^\beth\cong\Sp_\X^{-1}(Z)$.
}

\pa{ Assume from now on that we are working over a trivially valued field $k$, and let $\X$ be a special formal $k$-scheme.
The closed immersion $\X_0\to\X$ gives rise to an immersion $\big(\X_0\big)^\beth \to \X^\beth$, and we define the \defi{punctured Berkovich space} $\Xeta$ of $\X$ as the subspace
\[
\Xeta=\X^\beth\setminus\X_0^\beth
\]
of $\X^\beth$. 
It's a $k$-analytic space, introduced by Thuillier in \cite[1.7]{Thu07} (where it is called the generic fiber of $\X$).
Any adic morphism of special formal $k$-schemes $f:\Y\to\X$ induces a morphism $f^*\colon\Y^*\to\Xeta$, since we have 
$\big(f^\beth\big)^{-1}\big(\X_0^\beth\big)=\Y_0^\beth$.
We will denote again by $\Sp_\X\colon\Xeta\to\X$ the restriction of the specialization map to $\Xeta$.
}

\Pa{Examples}{
If $\X$ is of finite type over $k$, then $\X_0=\X$, and therefore $\Xeta$ is empty. 
If $\X=\Spf\big(k\lbrack\lbrack t\rbrack\rbrack\big)$, then $\Xeta$ is the punctured open unit disc in $\mathbb A^{1,\mathrm{an}}_k$, which is homeomorphic to the open interval $\left]0,1\right[$.
}

\pa{\label{example_affine_punctured_space}
If $\X=\Spf\big(k\{X_1,\ldots,X_n\}\lbrack\lbrack Y_1,\ldots,Y_m\rbrack\rbrack/(f_1,\ldots,f_r)\big)$ is an affine special formal $k$-scheme, we can describe $\Xeta$ along the lines of~\ref{example_affine_beth_space}.
The complement in $\Xeta$ of the zero locus $V(Y_i)$ of one of the $Y_i$'s is the increasing union $\Xeta\setminus V(Y_i)=\bigcup_{0<\varepsilon\leq1/2}W_{i,\varepsilon}$, where $W_{i,\varepsilon}$ is the subspace of $\X^\beth$ cut out by the inequalities $|Y_j|\leq1-\varepsilon$ for every $j$ and $\varepsilon\leq|Y_i|$.
The subspace $W_{i,\varepsilon}$ is an affinoid domain of $\Xeta$, with associated affinoid $k$-algebra
\[
\frac{k\big\{X_1,\ldots,X_n\big\}\big\{\varepsilon Y_i^{-1},(1-\varepsilon)^{-1}Y_1,(1-\varepsilon)^{-1}Y_2,\ldots,(1-\varepsilon)^{-1}Y_m\big\}}{(f_1,\ldots,f_r)}.
\]
We then have $\Xeta=\bigcup_{i=1}^m\big(\bigcup_{0<\varepsilon\leq1/2}W_{i,\varepsilon}\big)$.
Moreover, if we denote by $W_{i,\varepsilon}^\circ$ the open subspace of $\X^\beth$ cut out by $|Y_j|<1-\varepsilon$ for every $j$ and $\varepsilon>|Y_i|$, then the family $\big\{W_{i,\varepsilon}^\circ\big\}_{i=1,\ldots,m,\,0<\varepsilon\leq1/2}$ is an open cover of $\Xeta$.
Note that if $t$ is a nonzero element of an ideal of definition of $\X$, we can analogously write $\Xeta\setminus V(t)=\X^\beth\setminus V(t)$ as an increasing union of affinoid domains $\{W_{t,\varepsilon}\}_\varepsilon$.
}

\Pa{Fundamental example}{\label{R: interpretation algebraic case}
We now discuss in detail what happens in the algebraic case of \ref{example_algebraic}, when working over $k$.
Let $X$ be a separated $k$-scheme of finite type.
Then with every point $x$ of $X^{\mathrm{an}}$ we can associate a morphism $\varphi_x\colon\Spec\big(\rescompl{x}\big)\to X$, which sits in the following commutative diagram:
\begin{displaymath}
    \xymatrix@R=1.5pc@C=3pc@M=3pt@L=3pt{
        \Spec\big(\rescompl{x}\big) \ar[d] \ar[r]^(.61){\varphi_x} & X\ar[d] \\
        \Spec\big(\rescompl{x}^\circ\big) \ar[r]       & \Spec(k) }
\end{displaymath}
where $\rescompl{x}^\circ$ is the valuation ring of $\rescompl{x}$.
We say that $x$ \defi{has center} on $X$ if we can fit in the diagram above a morphism $\overline{\varphi_x}\colon\Spec\big(\rescompl{x}^\circ\big)\longrightarrow X$ that extends $\varphi_x$.
By the valuative criterion of separatedness if such an extension exists then it is unique.
The \defi{center} of $x$ on $X$ is then by definition the image in $X$ of the closed point of $\Spec(\rescompl{x}^\circ)$ via $\overline{\varphi_x}$.
We denote this point by $\Sp_X(x)$, and we write $X^\beth$ for the subset of $X^{\mathrm{an}}$ consisting of the points that have center on $X$.
The space $X^\beth$ is a compact analytic domain of $X^\mathrm{an}$ which can be thought of as a bounded version of $X^{\mathrm{an}}$, and it coincides with the space defined in \ref{definition_beth_space} if $X$ is seen as a formal $k$-scheme of finite type.
For example, $\big(\mathbb A^n_k\big)^\beth$ is the closed unit polydisc in $\mathbb A^{n,\mathrm{an}}_k$.
If $X$ is proper, then by the valuative criterion of properness we have $X^\beth\cong X^{\mathrm{an}}$.
Now let $Z$ be a closed subvariety of $X$ and set $\X=\widehat{X/Z}$. 
Then we have $\X^\beth=\Sp_X^{-1}(Z)$.
Moreover, the restriction of $\Sp_X$ to $\X^\beth$ is the specialization map $\Sp_\X$ defined in \ref{definition_specialization_map}.
The space $\X^\beth$ can be thought of as an (infinitesimal) \defi{tubular neighborhood} of $Z^\beth$ in $X^\beth$.
Note that $Z^{\mathrm{an}}$ is canonically isomorphic to the subspace $\rho^{-1}(Z)$ of $X^{\mathrm{an}}$, where $\rho\colon X^{\mathrm{an}}\to X$ is the structure morphism defined in Example~\ref{analytification_functor}.
Since $Z$ is closed in $X$, by the valuative criterion of properness we have $Z^\beth=Z^{\mathrm{an}}\cap X^\beth\subset X^{\mathrm{an}}$.
Therefore, we have $\Xeta=\X^\beth\setminus Z^\beth = \Sp_X^{-1}(Z)\setminus\rho^{-1}(Z)$.
In words, $\Xeta$ is the set of semivaluations on $X$ that have center in $Z$ but are not semivaluations on $Z$.
It can be thought of as a \defi{punctured tubular neighborhood} (or \defi{link}) of $Z^\beth$ in $X^\beth$.
}

\pa{\label{invariance_generic_fiber}
Let $f:\Y\to\X$ be an admissible blowup of special formal $k$-schemes. Then $f$ induces an isomorphism of punctured spaces 
$f^*\colon\Y^*\stackrel{\sim}{\longrightarrow}\Xeta$. 
In the algebraic case of Example~\ref{R: interpretation algebraic case} this follows from the valuative criterion of properness (see \cite[1.11]{Thu07}); the general case is \cite[4.5.1]{Ben-BassatTemkin2013}.
}

We conclude the section by giving definitions of admissibility for special formal $k$-schemes and for special $k$-algebras.

\pa{\label{definition_admissible_scheme}
Let $\X$ be a special formal $k$-scheme. 
We say that $\X$ is \defi{admissible} if the canonical morphism of sheaves
$
\OO_\X\to (\Sp_\X)_*\OO_\Xeta
$
is a monomorphism.
This is equivalent to the fact that $\OO_\X(U)\to\OO_\Xeta\big(sp_\X^{-1}(U)\big)$ is injective for every open $U$ of $\X$, so the property of being admissible can be thought of as having schematically dense generic fiber.
}

\pa{\label{definition_admissible_algebra}
If $A$ is a special $k$-algebra and $J$ is the largest ideal of definition of $A$, we define the \defi{torsion ideal} of $A$ as $A_{\tors}=\big\{a\in A \,\,\,\big|\,\,\,a\in A_{t-\tors}\,\, \forall t\in J\big\}$, where $A_{t-\tors}$ denotes the $t$-torsion of $A$; then $A_{\tors}$ is an ideal of $A$.
We say that $A$ is \defi{admissible} if $A_{\tors}=0$.
}

\Pa{Remarks}{\label{admissible_equivalence}
If $A$ is a nonzero admissible special $k$-algebra, then the largest ideal of definition of $A$ is nonzero, so that $A$ is not topologically of finite type over $k$.
If moreover $A$ is a domain, the converse holds as well: $A$ is admissible if and only if it is not topologically of finite type over $k$.
If $\{g_1,\ldots,g_s\}$ is a set of generators of $J$, then $A_{\tors}=\cap_{i=1}^sA_{g_i-\tors}$, hence $A$ is admissible if and only if the canonical morphism $A\to\prod_{i=1}^s A[g_i^{-1}]$ is injective.
As is done in the finite type case in \cite[7.3.13]{Bosch14}, one can use this injectivity to deduce that, if $A$ is an admissible special $k$-algebra, then for every element $f$ of $A$ the complete localization $A\big\{f^{-1}\big\}$ is admissible. 
If $A$ is an algebra topologically of finite type over $k[[t]]$,
seen as a special $k$-algebra, then $A$ is admissible if and only if it has no $t$-torsion. 
This shows that our definition of admissible algebra coincides with the usual one in this case.
We will prove in Proposition~$\ref{admissible}$ that an affine special formal $k$-scheme $\Spf A$ is admissible if and only if $A$ is an admissible special $k$-algebra. 
It will then be clear that our definition is analogous to the usual one for formal $R$-schemes of finite type.
}


\section{Normalized Berkovich spaces of special formal \texorpdfstring{$k$-} -schemes}
\label{section_3.1}

In this section we start by defining an $\I$-action on the punctured Berkovich space $\Xeta$ of a special formal $k$-scheme $\X$. 
We then introduce our primary object of study, the Normalized Berkovich space $\T$ of $\X$, as the quotient of $\Xeta$ by this action.

\pa{\label{ex_action_closed_disc}
One important feature of Berkovich analytic spaces is that they distinguish between equivalent but not equal seminorms.
For example, if $k$ is a trivially valued field, $\gamma$ is an element of $\I$, and $|\cdot|_x$ is an element of the closed unit disc $\Spf(k\{T\})^\beth=\m{k\{T\}}$ in the analytic affine line $\mathbb A^{1,\mathrm{an}}_k$, then also $|\cdot|_x^\gamma$ is an element of $\m{k\{T\}}$.
Indeed, the Banach norm of $k\{T\}=k[T]$ is the $T$-adic one with $|T|=1$, so it is the trivial norm; it follows that the elements of $\m{k\{T\}}$ are the seminorms $|\cdot|_x$ on $k\{T\}$ satisfying $|f|_x\leq 1$ whenever $f\in k\{T\}$. 
Then $|\cdot|_x^\gamma$ is multiplicative, trivial on $k$, it satisfies both the ultrametric inequality and $|f|_x^\gamma\leq 1$ for $f\in k\{T\}\setminus\{0\}$. 
Similarly, if $\X=\Spf(k\lbrack\lbrack T\rbrack\rbrack)$ then $\X^{\beth}$ is the open unit disc $D_-$  in the analytic affine line $\mathbb A^{1,\mathrm{an}}_k$ and $\Xeta$ is the punctured open unit disc $D_-\setminus\{0\}$.
The latter is homeomorphic to the open segment $\left]0,1\right[$, and under this identification $\I$ acts freely on it by exponentiation. 
Observe that the fact that the absolute value of $k$ is trivial, and thus invariant under exponentiation by elements of $\I$, is crucial.
}

\pa{
More generally, let $k$ be a trivially valued field and consider an affine special formal $k$-scheme $\X=\Spf\big(k\{X_1,\ldots,X_n\}\lbrack\lbrack Y_1,\ldots,Y_m\rbrack\rbrack/(f_1,\ldots,f_r)\big)$.
It follows from the definition given in \ref{definition_beth_space} that, by seeing it as a subset of the analytic affine space $\mathbb A_k^{n+m,\mathrm{an}}$, the set $\X^\beth$ is the set of multiplicative seminorms $|\cdot|_x \colon k[X_1,\ldots,X_n] \lbrack\lbrack Y_1,\ldots,Y_m \rbrack\rbrack /(f_1,\ldots,f_r)$ which are trivial on $k$ and such that $|X_i|_x\leq1$ and $|Y_j|_x<1$ for every $i$ and $j$.
Therefore, for every element $\gamma$ of $\I$ the seminorm $|\cdot|_x^\gamma$ is itself an element of $\X^\beth$.
Moreover, $\X_0^\beth$ is defined in $\X^\beth$ by the equalities $Y_1=\ldots=Y_m=0$, therefore the $\I$-action restricts to an action on $\X^*=\X^\beth\setminus \X_0^\beth$.
}

\pa{
If $\Spf B$ is a subscheme of the affine special formal scheme $\Spf A$, the induced map $(\Spf B)^*\to(\Spf A)^*$, being induced by the composition of a seminorm with the morphism $B\to A$, is equivariant with respect to the $\I$-actions.
This allows to extend the $\I$ action to the punctured Berkovich space of a general special formal $k$-scheme $\X=\bigcup_i\Spf(A_i)$, by covering $(\Spf A_i)\cap(\Spf A_j)$ with affine subschemes.
}

\begin{rem}The $\I$-action on $\Xeta$ is free (i.e. the orbits $\I\cdot x$ are in bijection with $\I$). 
Indeed, either $\I\cdot x\cong \I$ or $\I\cdot x=\{x\}$; the latter is equivalent to $x$ being a trivial absolute value, but all trivial absolute values of $\X^{\beth}$ lie in $\X_0^{\mathrm{an}}$.
\end{rem}

\pa{
We deduce from the discussion above that the association $\X\mapsto\Xeta$ gives a functor from the category of special formal $k$-schemes with adic morphisms to the category of $k$-analytic spaces with a free $\I$-action on the underlying topological space and equivariant analytic morphisms.
}

We now consider the quotient of $\Xeta$ by the $\I$-action.

\pa{\label{def_quotient_map}
Let $\X$ be a special formal scheme over $k$. 
Denote by $\T$ the (set-theoretic) quotient of the space $\Xeta$ by the action of $\I$, and by $\pi:\Xeta\to\T$ the quotient map. 
We endow $\T$ with both the quotient topology and the quotient $G$-topology. 
The latter is defined as follows: we declare that a subspace $U$ of $\T$ is $G$-admissible if $\pi^{-1}(U)$ is an analytic domain of $\Xeta$, and that a family $\{U_i\}_i$ of subspaces of $U$ is a $G$-cover of $U$ if $\{\pi^{-1}(U_i)\}_i$ is a $G$-cover of $\pi^{-1}(U)$. 
It is easy to verify that this defines a $G$-topology on $\T$, which is finer than the quotient topology because the $G$-topology on $\Xeta$ is finer than the Berkovich topology. 
As is the case for usual Berkovich spaces, the $G$-admissible subsets of $\T$, which should not be thought of as open subsets, will be called \defi{analytic domains} of $\T$.
To provide $\T$ with the structure of a ringed $G$-topological space in $k$-algebras, we endow it with the sheaf $\OO_\T=\pi_*\OO_\Xeta$, the push-forward of the sheaf of analytic functions on $\Xeta$ via the projection map.
We will often denote by $\OO_\T$ also the restriction of the previous sheaf to the usual topology of $\T$, and when talking about stalks of $\OO_\T$ we will always consider the stalk with respect to the usual topology.
}

\begin{rem}
The $G$-covers of $\T$ can be described explicitly as follows.
If $U$ is an analytic domain of $\T$ and $\{U_i\}_{i\in I}$ is the family consisting of the analytic domains of $\T$ contained in $U$, then the $U_i$ form a $G$-cover of $U$ if and only if for every point $x$ of $U$ there exists a finite subset $I_x$ of $I$ such that $\bigcup_{i\in I_x}U_i$ is a neighborhood of $x$ and $x\in\bigcap_{i\in I_x}U_i$.
This fact follows from the corresponding statement for $\Xeta$ and from the openness of $\pi$, which will be proven in Lemma~\ref{L: normalized hausdorff}.
\end{rem}

\pa{ \label{bounded_functions_on_T} 
If $f\in\OO_T(V)$ is a function on $V\subset\T$, we do not obtain a real value by evaluating $f$ in a point of $V$. 
Nevertheless, it makes sense to ask whether this value lies in $\{0\}$, $\{1\}$, $\left]0,1\right[$ or $]1,\infty[$, since these sets are the orbits of the action of $\I$ on $\mathbb R_{\geq0}$ by exponentiation. 
In particular, the sheaf $\pi_*\OOO_\Xeta$, which is a subsheaf of $\OO_\T$, can really be thought of as the \defi{sheaf of analytic functions bounded by $1$} on $T_\X$, and the sheaf $\pi_*\OOOO_\Xeta$ can be thought of as the \defi{sheaf of analytic functions strictly bounded by $1$} on $T_\X$. 
We denote these sheaves by $\OOO_\T$ and $\OOOO_\T$ respectively.
}

\pa{\label{abstract_valuation_ring}
A more intrinsic way to see this is to observe that with any point $x$ of $\T$ is associated an abstract valued field, the field $\mathscr H(y)$ for any point $y$ of $\pi^{-1}(x)$, endowed with an abstract valuation but not with an absolute value.
If $f$ is a function on $\T$ then $f(x)$ makes sense as an element of this valued field, and $f$ is bounded by $1$ at $x$ if $f(x)$ belongs to the corresponding abstract valuation ring.
Moreover, $|f(x)|$ makes sense as an element of the corresponding value group, while choosing a preimage $y$ for $x$ corresponds to choosing an embedding of this value group into $\R_+^\times$, which yields a real value for $|f(x)|$.
}

\begin{rem} \label{functions_on_T} 
It also makes sense to evaluate at points of $\T$ every function that is constant on the orbits of points for the $\I$-action. 
For example, if $f$ and $g$ are functions on $V\subset\T$ and $x$ is a point of $V$ where $f$ and $g$ do not vanish, then we can evaluate $\log|f|/\log|g|$ at $x$. 
Note that this value is encoded in the structure of the normalized space.
For example, if both $f$ and $g$ take values in $\left]0,1\right[$, then $(\log|f|/\log|g|)(x)$ can be defined as $\sup\{a/b \,\,|\,\, a,b\in\mathbb N,b\neq0 \mbox{ and } |f^b(x)|\leq|g^a(x)|\}$; the other cases are similar.
\end{rem}

\begin{lem}\label{L: normalized hausdorff}
The projection $\pi:\Xeta\to\T$ is an open map, and the topological space underlying $\T$ is Hausdorff.
\end{lem}

\begin{proof}
The projection $\pi:\Xeta\to\T$ is an open map because $I$ acts on $\Xeta$ by homeomorphisms (this follows from the definition of the Berkovich topology on $\Xeta$ and the continuity of $x\mapsto x^\lambda$). 
Since $\Xeta$ is Hausdorff, to prove that the quotient $\T$ is Hausdorff as well it is sufficient to show that the orbit equivalence relation is closed in $\Xeta\times\Xeta$.
Moreover, without loss of generality we can assume that $\X=\Spf A$ is affine.
Now, if $x$ and $y$ be two points of $\Xeta$ that are not in the same $I$-orbit, we can find two functions $f$ and $g$ in $A$, not vanishing on $x$ and $y$, and such that $\log|f(x)|/\log|g(x)|\neq \log|f(y)|/\log|g(y)|$. 
Since the quotients of the logarithms are continuous and constant on orbits, the orbit equivalence relation is closed in $\Xeta\times\Xeta$, therefore $\T$ is Hausdorff.
\end{proof}

\begin{lem}
The restriction of the $G$-sheaf $\OO_\T$ to the usual topology of $\T$ is a local sheaf.
\end{lem}

\begin{proof}
Let $x$ be a point of $\T$.
Then $\calI=\big\{f\in\OO_{\T,x} \mbox{ s.t. }|f(x)|=0\big\}$ is an ideal of $\OO_{\T,x}$, where we write $\OO_{\T,x}$ for the stalk of $\OO_{\T}$ in $x$ for the usual topology. 
If $|f(x)|\neq 0$ then $f$, seen as a function on a neighborhood of $\pi^{-1}(x)$ in $\Xeta$, does not vanish in any point of some open neighborhood $U$ of some point of $\pi^{-1}(x)$.
Therefore $f$ has no zero, and is hence invertible, on the $\I$-invariant subspace $\pi^{-1}(\pi(U))$, which is open by Lemma~\ref{L: normalized hausdorff}.
This proves that $\calI$ is the unique maximal ideal of $\OO_{\T,x}$, and so the restriction of $\OO_{\T}$ to the usual topology is local.
\end{proof}

\begin{ex} \label{valtree}
If $\X$ is $\Spf\big(\mathbb C[[X,Y]]\big)$, the completion of the complex affine plane $\mathbb A^2_{\mathbb C}$ at the origin, the topological space $\T$ is canonically homeomorphic to the \defi{valuative tree} $T$ introduced by Favre and Jonsson in \cite{valtree}. 
The valuative tree is defined as the set of (semi-)valuations $v$ on $\mathbb C[[X,Y]]$ extending the trivial valuation on $\mathbb C$ and such that $\min\{v(X),v(Y)\}=1$, endowed with the topology it inherits from the Berkovich space $\Xeta$ via the inclusion that sends a valuation $v\in T$ to the seminorm $e^{-v}$. 
The restriction of the projection $\pi:\Xeta\to\T$ to $T$ is then a continuous bijection, therefore it is a homeomorphism because $T$ is compact by \cite[5.2]{valtree} and $\T$ is Hausdorff by Lemma~\ref{L: normalized hausdorff}.
\end{ex}

\begin{ex}\label{example_explicit_normalization}
More generally, in the algebraic case discussed in Examples~\ref{example_algebraic} and \ref{R: interpretation algebraic case}, i.e. when $\X=\widehat{X/Z}$ is the formal completion of a $k$-variety $X$ along a closed subvariety $Z$, the normalized Berkovich space $\T$ can be thought of as the {\it normalized non-archimedean link} of $Z$ in $X$.
The topological space underlying $\T$ can be described explicitly as the space of normalized valuations on $X$ that are centered on $Z$ but are not valuations on $Z$, i.e. $\T=\left(\Sp_X^{-1}(Z)\setminus Z^\beth\right)\big/\I$.
An explicit normalization can be given as follows.
Let $\calI$ be the coherent ideal sheaf of $X$ defining $Z$, and for each element $x$ of $\Xeta$ set $x(\calI)=\max\big\{|f(x)|\;\big|\;f\in\calI_{\Sp_\X(x)}\big\}>0$, where $\calI_{\Sp_\X(x)}$ denotes the stalk of $\calI$ at ${\Sp_\X(x)}$.
Then, as in Example~\ref{valtree}, since for every $\gamma$ in $\I$ and $x$ in $\Xeta$ we have $\gamma\cdot x(\calI)=x(\calI)^\gamma$ the restriction of $\pi$ to the subspace $\{x\in X^\beth \;|\; x(\calI)=1/e\}$ of $\Xeta$ is a homeomorphism onto $\T$.
Therefore in this case the topological space we consider is similar to the \defi{normalized valuation space} considered in \cite[\S2.2]{BoucksomdeFernexFavreUrbinati13}, but they consider only valuations with trivial kernel.
Moreover, our definition as a quotient is more intrinsic as it does not depend on the choice of the real number $1/e$.
\end{ex}

\pa{
Let $\X$ be a special formal $k$-scheme.
Then the specialization map on $\Xeta$ induces an anticontinuous map $\Sp_\X:\T\to\X$, which we call again \emph{specialization}.
Indeed, as we observed in \ref{abstract_valuation_ring} the valuation ring of two elements of $\Xeta$ in the same $\I$-orbit is the same, therefore the specialization map on $\Xeta$ passes to the quotient, inducing a map of sets $\Sp_\X:\T\to\X$ such $\Sp_\X\circ\pi=\Sp_\X$. 
Anticontinuity follows from the fact that the specialization on $\Xeta$ is anticontinuous and $\pi$ is open by Lemma~\ref{L: normalized hausdorff}.
We will prove in Proposition~\ref{sp_surjective} that if $\X$ is admissible then the specialization map is surjective.
}

\pa{\label{definition_category_C} 
To study the spaces we have been considering so far, it is convenient to temporarily introduce a suitable category.
We denote by $\calC$ the category whose objects are the triples $\big(T,\OO_T,\OOO_T\big)$, where $T$ is a topological space endowed with an additional $G$-topology that is finer than its given usual topology (which means that every open subspace of $T$ is also a $G$-admissible open and every open cover is a $G$-cover), $\OO_T$ is a sheaf of $k$-algebras on $T$ for the $G$-topology, and $\OOO_T$ is a subsheaf in $k$-algebras of $\OO_T$; and such that a morphisms $\big(T,\OO_T,\OOO_T\big)\to \big(T',\OO_{T'},\OOO_{T'}\big)$ is given by a continuous and $G$-continuous map $f:T\to T'$ and a morphism of sheaves $f_\#:\OO_{T'}\to f_*\OO_{T}$ such that $f\big(\OOO_{T'}\big)\subset f_*\OOO_T$ and inducing a local morphism of local sheaves once restricted to the usual topology.
We will always write simply $\OO_T$ (respectively $\OOO_T$) also for the restriction of $\OO_T$ (resp. $\OOO_T$) to the topology of $T$ and, when no risk of confusion will arise, we will write $T$ for an object $\big(T,\OO_T,\OOO_T\big)$ of $\calC$. 
}

\pa{
Let $\X$ be a special formal scheme over $k$. 
We define the \defi{normalized Berkovich space} of $\X$ as the object $\T=\big(\T,\OO_\T,\OOO_\T\big)$ of $\calC$.
This gives a functor $T:\big(SFor_k\big)\to\calC$ from the category of special formal $k$-schemes with adic morphisms to $\calC$. 
In Section~\ref{section_3.4} we will investigate the properties of the functor $T$ and determine its essential image.
}

\begin{rem}\label{homotopy_normalized}
Thuillier proved in \cite{Thu07} that whenever $k$ is perfect, $X$ is a $k$-variety with singular locus $Z$ and $\X=\widehat{X/Z}$, the homotopy type of $\X^*$ is the same as the homotopy type of the dual complex $\mathrm{Dual}(D)$ of the exceptional divisor $D$ of a log resolution $Y$ of $X$.
Using toroidal methods, he constructs an embedding of $\mathrm{Dual}(D)\times\I$ into $\Y^*$ and a deformation retraction of the latter onto the former, where $\Y=(\widehat{Y/D})$.
As both the embedding and the retraction are $\I$-equivariant, quotienting by the action of $\I$ we obtain a deformation retraction of $T_\Y$ onto a closed subspace homeomorphic to $\mathrm{Dual}(D)$.
Since $T_\Y\cong\Txz$, we deduce that the homotopy type of $\Txz$ is the same as the homotopy type of $\mathrm{Dual}(D)$.
Note that by \cite{Kollar13} the homotopy type of $\mathrm{Dual}(D)$ can be almost arbitrary.
However, by \cite{deFernexKollarXu12} $\mathrm{Dual}(D)$ is contractible for a wide class of singularities, namely isolated log terminal singularities (in particular, for all toric or finite quotient singularities).
\end{rem}

\pa{\label{def_forgetful_functor}
We also have a forgetful functor $\forg:\big(An_{k((t))}\big)\to\calC$ sending a $k((t))$-analytic space $X$ to the triple $\big(X,\OO_X,\OOO_X\big)$. 
}


\section{Local analytic structure}
\label{section_3.2}
In this section we prove one of the main properties of normalized spaces of special formal $k$-schemes. 
Although those spaces are not analytic spaces themselves, as locally ringed spaces in $k$-algebras they are $G$-locally isomorphic to analytic spaces defined over some Laurent series field $k((t))$. 
This is the content of Corollary~\ref{locally_analytic}. 
This result is conceptually similar to those discussed in \cite[\S4.2, 4.3, 4.4, 4.6]{Ben-BassatTemkin2013}.
We deduce an analogue for normalized spaces of a theorem of de Jong (Corollary~\ref{dejong_normalized}), and a characterization of admissible special formal $k$-schemes (Proposition~\ref{admissible}).
We also prove that the specialization map is surjective for the normalized space of an admissible special formal $k$-scheme (Theorem~\ref{sp_surjective}).
We first need some results about the Berkovich spaces associated with affine special formal $k$-schemes.

\begin{prop}\label{G-admissible_strip} 
Let $\X$ be an affine special formal scheme over $k$, let $t$ be a nonzero element of an ideal of definition of $\X$ and let $U$ be a subset of $\Xeta\setminus V(t)$ stable under the action of $\I$. 
Then $U$ is an analytic domain of $\Xeta$ if and only if we can write it as a union $U=\cup_i U_i$ in such a way that each $U_i$ is stable under the action of $\I$ and is an increasing union $U_i=\cup U_{i,\varepsilon}$ for $\varepsilon$ small enough, with $U_{i,\varepsilon}$ a strict affinoid subdomain of $W_{t,\varepsilon}$, and $\{U_{i,\varepsilon}\}_{i,\varepsilon}$ is a $G$-cover of $U$.
\end{prop}

Before proving this proposition we will establish a simple lemma.

\begin{lem}\label{orbit_affinoid}
Let $\X$ be an affine special formal scheme over $k$, let $t$ be a nonzero element of an ideal of definition of $\X$ and let $V$ be an affinoid domain of $\Xeta$ such that $|t|=1/2$ on $V$.
Denote by $\I\cdot V$ the set of translates of $V$ in $\Xeta$ under the $\I$-action. 
Then we can write $\I\cdot V$ as a finite union $\cup_{i}V_i$, with each $V_i$ stable under the action of $\I$ and such that $V_{i,\varepsilon}:=V_i\cap W_{t,\varepsilon}$ is a strict affinoid subdomain of $W_{t,\varepsilon}$ for $\varepsilon$ small enough.
In particular, each $V_i$ is an increasing union of affinoid domains of $\Xeta$, and $\I\cdot V$ is an analytic domain of $\Xeta$.
\end{lem}

\begin{proof}
For every $\varepsilon\leq 1/2$, $V$ is an affinoid subdomain of $W_{t,\varepsilon}$. 
Therefore, by Gerritzen-Grauert theorem (proofs valid in the case of a trivially valued field are given in \cite{Ducros03} and \cite{Temkin05}), $V$ is a finite union $V=\cup V'_i$ of rational domains of $W_{t,\varepsilon}$. 
Each $V'_i$ is by definition determined in $W_{t,\varepsilon}$ by finitely many inequalities $|f_j|\leq r_j|g_j|$, for some analytic functions $f_j,g_j$ on $W_{t,\varepsilon}$ and $r_j> 0$.
Then we have
\begin{align*}
& (\I\cdot V'_i)\cap W_{t,\varepsilon} = \big\{x\in W_{t,\varepsilon} \mbox{ s.t. }x^\lambda\in V'_i \mbox{ for some }\lambda\in\I\big\} \\ 
 = & \big\{x\in W_{t,\varepsilon} \mbox{ s.t. there exists }\lambda \mbox{ s.t. } |f_j(x)|\leq r_j^{\frac{1}{\lambda}}|g_j(x)|\mbox{ for all }j\big\} \\
 = & \big\{x\in W_{t,\varepsilon} \mbox{ s.t. }|f_j(x)|\leq |t(x)|^{a_j}|g_j(x)|\mbox{ for all } j\big\},
\end{align*}
where $a_j= - \log r_j / \log 2$, yielding a strict affinoid subdomain of $W_{t,\varepsilon}$.
Observe that in the last equality we used the fact that $|t|=1/2$ on $V'_i$, hence $|t(x)|^{a_j}=(1/2)^{\frac{a_j}{\lambda}}=r_j^{\frac{1}{\lambda}}$ for every $x$ in $(\I\cdot V'_i)\cap W_{t,\varepsilon}$ and every $j$.
%
Now, if we set $V_i=\I\cdot V'_i$ and $V_{i,\varepsilon}=V_i\cap W_{t,\varepsilon}$ for every $\varepsilon\leq 1/2$, then the $V_i$ and $V_{i,\varepsilon}$ satisfy our requirements.
Finally, $\I\cdot V$ is an analytic domain of $\Xeta$, which is $G$-covered by the affinoid domains $V_i\cap W_{t,\varepsilon}$. 
Indeed, every point of $\I\cdot V$ is contained in the interior of one set of the form $\I\cdot V\cap W_{t,\varepsilon}$, and the latter is the finite union of the affinoid domains $V_i\cap W_{t,\varepsilon}$.
\end{proof}

\begin{proof}[Proof of Proposition~\ref{G-admissible_strip}]
If $U$ is an analytic domain of $\Xeta$ then it is $G$-covered by the affinoid domains that it contains.
Then the intersection of each of them with $W_{t,1/2}$ is an affinoid domain on which $|t|=1/2$, and applying Lemma~\ref{orbit_affinoid} to all such intersections we get the decomposition that we want.
The converse implication is obvious, since $U$ is by definition $G$-covered by the affinoid domains $U_{i,\varepsilon}$.
\end{proof}

\pa{
For $r\in\left]0,1\right[$, we denote by $K_r$ the affinoid $k$-algebra $k\{rt^{-1},r^{-1}t\}$; it is the completed residue field $\rescompl{r}$ of the point $r$ of $\Spf(k\lbrack\lbrack t\rbrack\rbrack)^*\cong\left]0,1\right[$, and an easy computation shows that it is the field $k((t))$ with the $t$-adic absolute value such that $|t|=r$. 
For $0<\varepsilon<1/2$, the $k$-algebras $\calA_\varepsilon=k\{\varepsilon T^{-1},(1-\varepsilon)^{-1}T\}$ are also isomorphic to the field $k((t))$, but their Banach norms are not $t$-adic (they are not even absolute values). 
If $r\in[\varepsilon,1-\varepsilon]$, the identity map $\calA_{\varepsilon}\to K_r$ is a bounded morphism of Banach $k$-algebras. 
Its boundedness is easy to check algebraically; geometrically this corresponds to the inclusion of the point $r$, into $\calM(\calA_\varepsilon)$, seen as the annulus $\left[\varepsilon,1-\varepsilon\right]$ in $\Spf(k\lbrack\lbrack t\rbrack\rbrack)^*\cong\left]0,1\right[$.
Nevertheless, note that despite having different Banach norms $K_r$ and $\calA_{\varepsilon}$ are isomorphic not only as $k$-algebras, but also as topological $k$-algebras, since the neighborhoods of zero in both algebras coincide.
This has as a very important consequence the following result.
}

\begin{lem} \label{lemma_tensor}
Let $\calB$ be a strict affinoid algebra over $\calA_{\varepsilon}$ and let $r$ be an element of $\left[\varepsilon,1-\varepsilon\right]$. Then the canonical morphism $\calB\to \calB\hat\otimes_{\calA_{\varepsilon}}K_r$ is an isomorphism of $k$-algebras.
\end{lem}

\begin{proof} 
If $\calB=\calA_{\varepsilon}\{X_1,\ldots,X_n\}$, an elementary computation shows that the convergence conditions for a series of the form $\sum_{i_0,I}a_{i_0,I}t^{i_0}\underline X^I$, for $a_{i_0,I}\in k$, to belong to either $\calB$ or $\calB\hat\otimes_{\calA_{\varepsilon}}K_r$ are the same (namely, the coefficients $a_{i_0,I}$ have tend to zero for the $t$-adic topology, which is the same on $\calA_\epsilon$ and on $K_r$). 
Therefore $\calB\otimes_{\calA_{\varepsilon}}K_r$, which is isomorphic to $\calB$ as a $k$-algebra via the canonical morphism $\calB\to \calB\otimes_{\calA_{\varepsilon}}K_r$, is already complete with respect to the tensor product seminorm, and hence coincides with $\calB\hat\otimes_{\calA_{\varepsilon}}K_r$.
In the general case, that is $\calB=\calA_{\varepsilon}\{X_1,\ldots,X_n\}/I$, to show that $\calB\otimes_{\calA_{\varepsilon}}K_r$ is complete observe that $\calB\otimes_{\calA_{\varepsilon}}K_r\cong \big({\calA_{\varepsilon}\{X_1,\ldots,X_n\}\otimes_{\calA_{\varepsilon}}K_r}\big)\big/{I}$ as normed algebras. 
Indeed, they are canonically isomorphic as $k$-algebras, and the fact that the norms are the same follows from the fact that under the isomorphism we have $(a+I)\otimes t=(a\otimes t)+I$, and on both sides the norm of such an element is $|a|\cdot|r|$.
The algebra on the right hand side is complete since it is an admissible quotient of a Banach algebra by a closed ideal, hence $\calB\otimes_{\calA_{\varepsilon}}K_r$ is complete.
\end{proof}

\begin{rem}
The result above can fail if the affinoid $\calA_\varepsilon$-algebra $\calB$ is not strict because the first computation in the proof (that is the case $I=0$) breaks down. 
For example, if $0<r<s<\varepsilon\leq1/2$, then $K_r\hat\otimes_{\calA_\varepsilon}K_s=0$. 
Indeed, the tensor product seminorm on $K_r\otimes_{\calA_\varepsilon}K_s$ is the zero seminorm since the element $1\otimes 1$ of the tensor product is equal to $t^n\otimes t^{-n}$ for every $n\in\N$, so $|1\otimes 1|\leq r^ns^{-n}$, and the latter goes to zero as $n$ goes to infinity. 
\end{rem}

\pa{\label{projection_structure}
If $\X_t$ is a special formal scheme over $k[[t]]$ then it can be seen as a special formal scheme $\X$ over $k$, so we get a morphism of formal $k$-schemes $\X\to\Spf{(k[[t]])}$ and therefore a morphism of $k$-analytic spaces $f:\X^\beth\to\Spf{(k[[t]])}^\beth$.
If $r$ is any point of $\Spf(k\lbrack\lbrack t\rbrack\rbrack)^\beth\cong\left[0,1\right[$, then we can consider the fiber product of $\X^\beth$ with the point $r$ in the category of analytic spaces over $\Spf(k\lbrack\lbrack t\rbrack\rbrack)^\beth$ (see \cite{Ber93}). 
This analytic space is defined over the non-archimedean field $\rescompl{r}$, which coincides with $k$ whenever $r=0$ and is otherwise isomorphic to the field $k((t))$ with the $t$-adic absolute value such that $|t|=r$.
The topological space underlying this fiber product is canonically homeomorphic to the topological fiber: $
f^{-1}(r)\cong\X^\beth\times_{\Spf(k\lbrack\lbrack t\rbrack\rbrack)^\beth}\m{\rescompl{r}}$.
If $\X_t=\Spf A$ is affine, by seing points of both $\X_t^\beth$ (considered as an analytic space over the field $k((t))$ with the $t$-adic absolute value such that $|t|=r$) and $\X^\beth$ as seminorms on $A$, the points of $\X_t^\beth$ satisfying the additional condition $|t|=r$, we deduce the existence of a map $\X_t^\beth\to\X^\beth$, which factors through $f^{-1}(r)$.
In the general case, those maps glue to a map $\X_t^\beth\to f^{-1}(r)$.
}

\begin{lem}\label{lemma_omeo} 
Let $\X_t$ and $f:\X^\beth\to\Spf{(k[[t]])}^\beth$ be as above, choose $0<r<1$ and endow $k((t))$ with the $t$-adic absolute value such that $|t|=r$. 
Then the map $\X_t^\beth\to f^{-1}(r)$ constructed above induces an isomorphism of $k((t))$-analytic spaces between $\X^\beth_t$ and $f^{-1}(r)$.
\end{lem}

\begin{proof}
We follow the lines of \cite[4.3]{Nicaise11}. 
Since $\X_t^\beth$ is $G$-covered by the Berkovich spaces of the affine formal subschemes of $\X_t$, we can assume that $\X_t$ is affine, with associated $k[[t]]$-algebra $k[[t]]\{X_1,\ldots,X_n\}[[Y_1,\ldots,Y_m]]/I$. 
Then following \ref{example_affine_beth_space} we write $\X_t^\beth$ as an increasing union $\X_t^\beth=\bigcup_{0<\varepsilon\leq1/2}U_\varepsilon$, where $U_\varepsilon$ is the subspace of $\X_t^\beth$ cut out by the inequalities $|Y_i|\leq1-\varepsilon$; it is an affinoid domain with associated affinoid $k((t))$-algebra
$$
\frac{k((t))\{X_1,\ldots,X_n\}\{(1-\varepsilon)^{-1}Y_1,\ldots,(1-\varepsilon)^{-1}Y_m\}}{I}.
$$
Similarly, by \ref{example_affine_punctured_space} we can write $\Xeta\setminus V(t)$ as an increasing union $\Xeta\setminus V(t)=\bigcup_{0<\varepsilon\leq1/2}W_{t,\varepsilon}$, where $W_{t,\varepsilon}$ is the subspace of $\Xeta$ cut out by $\varepsilon\leq|t|\leq1-\varepsilon$ and $|Y_i|\leq1-\varepsilon$; it is an affinoid domain of $\Xeta$ with associated affinoid $k$-algebra
\[
\frac{{k\{\varepsilon t^{-1},(1-\varepsilon)^{-1}t\}\{X_1,\ldots,X_n\}\{(1-\varepsilon)^{-1}Y_1,\ldots,(1-\varepsilon)^{-1}Y_m\}}}{I}.
\]
For every $\varepsilon$ such that $\varepsilon<r<1-\varepsilon$, we get a map $U_\varepsilon\to W_{t,\varepsilon}$ that is a homeomorphism onto $f^{-1}(r)\cap W_{t,\varepsilon}$ and that coincides with the restriction  to $U_\varepsilon$ of the map $\X_t^\beth\to f^{-1}(r)$ defined in \ref{projection_structure}. 
We deduce that $\X_t^\beth\to f^{-1}(r)$ is a homeomorphism , and it induces an isomorphism of $k((t))$-analytic spaces because, since $f^{-1}(r)\cap W_{t,\varepsilon}$ can be identified with the subspace of $f^{-1}(r)$ defined by the inequalities $|Y_i|\leq1-\varepsilon$, which is an affinoid domain with associated affinoid $k((t))$-algebra $\calO_{W_{t,\varepsilon}}(W_{t,\varepsilon})\hat\otimes_{\calA_\varepsilon}\rescompl{r}\cong\calO_{U_\varepsilon}(U_\varepsilon)$, it identifies the affinoid domains of $\X_t^\beth$ with the ones of $f^{-1}(r)$.
\end{proof}

\pa{ \label{projection_structure2}
 The morphism $\X\to\Spf{(k[[t]])}$ of \ref{projection_structure} is adic if and only if $\X_t$ is locally of finite type over $k[[t]]$, so in general it does not induce a morphism between punctured Berkovich spaces.
However, we get an $\I$-equivariant morphism of $k$-analytic spaces $f|_{\X^\beth\setminus V(t)}:{\X^\beth\setminus V(t)}\to\Spf(k[[t]])^*$.
Since $f|_{\X^\beth\setminus V(t)}$ is equivariant and $\Spf\big(k[[t]]\big)^*\cong\left]0,1\right[$ consists of a unique $\I$-orbit, for any $r\in]0,1[$ we have a homeomorphism $f^{-1}(r)\cong \big(\X^\beth\setminus V(t)\big)/(\I)$, and the latter is homeomorphic to ${\T}\setminus V(t)$.
As a consequence of the results above, we obtain the following important theorem, which is purely formal and relies essentially on the structure of analytic domains obtained in Proposition~\ref{G-admissible_strip} and on the computation of Lemma~\ref{lemma_tensor}.
Recall that we have a quotient map $\pi:\Xeta\to\T$, defined in \ref{def_quotient_map}, and a forgetful functor $\forg:\big(An_{k((t))}\big)\to\calC$, discussed in \ref{def_forgetful_functor}.
}

\begin{thm} \label{thm_structure_arboretum}
Let $\X_t$ be a special formal scheme over $k[[t]]$ and choose $0<r<1$. 
Then, via the isomorphism of Lemma~\ref{lemma_omeo}, $\pi|_{f^{-1}(r)}:f^{-1}(r)\to\T\setminus V(t)$ induces an isomorphism between $\forg(\X^\beth_t)$ and $\T\setminus V(t)$ in $\calC$.
\end{thm}

\begin{proof}
Without loss of generality we can assume that $\X_t$ is affine and $\X_t^\beth$ is nonempty (if $\X_t^\beth$ is empty then $t$ is identically zero on $\Xeta$, hence $\T\setminus V(t)$ is empty as well).
We endow $k((t))$ with the $t$-adic absolute value such that $|t|=r$ and identify $\X_t^\beth$ with $f^{-1}(r)$ via Lemma~\ref{lemma_omeo} (although it will follow from the theorem that $\forg(\X_t^\beth)$ will not depend on the choice of $r$).
We will prove that the continuous map $\varphi:\X^\beth_t\stackrel{\sim}{\longrightarrow} \T\setminus V(t)$ obtained from Lemma~\ref{lemma_omeo} and \ref{projection_structure2} can be upgraded to an isomorphism in $\calC$. 
Observe that, once $\X_t^\beth$ is identified with $f^{-1}(r)$, the map $\varphi$ is induced by the projection $\pi$, and therefore it is a morphism of locally ringed $G$-topological spaces.
The map $f:\X^\beth\to\Spf\big(k[[t]]\big)^\beth$ of~\ref{projection_structure} is equivariant and therefore, since $\X_t^\beth$, and hence $\Xeta\setminus V(t)$, is nonempty, for every $0<\varepsilon<1/2$ it induces a surjective morphism from $W_{t,\varepsilon}$ to the affinoid domain $A_\varepsilon=\m{\calA_\varepsilon}\cong\left[\varepsilon,1-\varepsilon\right]$ of $\big(\Spf(k[[t]])\big)^\beth\cong\left[0,1\right[$.
The corresponding morphism $\calA_\varepsilon\to\calW_{t,\varepsilon}$ is the unique $k$-morphism sending $t$ to $t$, and it endows $\calW_{t,\varepsilon}$ with the structure of a strict affinoid algebra over $\calA_\varepsilon$.
To see that $\varphi$ is $G$-continuous we have to show that $\pi^{-1}(U)\cap \X_t^\beth$ is an analytic domain of $\X^\beth_t$ whenever $U\subset\T\setminus V(t)$ is such that $\pi^{-1}(U)$ is an analytic domain of $\Xeta$.
Using Proposition~\ref{G-admissible_strip}, we write $\pi^{-1}(U)$ as $\cup U_i$, with $U_{i,\varepsilon}=U_i\cap W_{t,\varepsilon}$ strict affinoid subdomain of $W_{t,\varepsilon}$ for every $i$ and every $\varepsilon$ small enough. 
Following the isomorphism of Lemma~\ref{lemma_omeo}, we deduce that $\pi^{-1}(U)\cap \X_t^\beth$ is $G$-covered by the affinoid domains $U_{i,\varepsilon}\times_{A_\varepsilon}\m{K_r}$ of $\X_t^\beth$, and is therefore an analytic domain of $\X^\beth_t$.
Moreover, if we choose $\varepsilon$ small enough, each $U_{i,\varepsilon}$, being strict over $W_{t,\varepsilon}$ that is strict over $A_\varepsilon$, is itself strict over $A_\varepsilon$. 
We deduce that, as $k$-algebras,
\begin{align*}
\OO_\Xeta(U_i)& = \varprojlim_\varepsilon\OO_\Xeta(U_{i,\varepsilon}) = \varprojlim_\varepsilon \big(\OO_\Xeta(U_{i,\varepsilon})\hat\otimes_{\calA_\varepsilon}K_r\big)
\\
& = \varprojlim_\varepsilon\OO_{\X_t^\beth}\big(U_{i,\varepsilon}\cap \X_t^\beth\big)
\\
& = \OO_{\X_t^\beth}\big(U_i\cap \X_t^\beth\big),
\end{align*}
where the second equality is given by Lemma~\ref{lemma_tensor}.
Since the $U_i$ form a $G$-cover of $U$, it follows that we have an isomorphism of $k$-algebras
$$
\OO_{\T\setminus V(t)}(U)\cong\OO_{\Xeta}\big(\pi^{-1}(U)\big)\cong\OO_{\X^\beth_t}\big(\pi^{-1}(U)\cap \X_t^\beth\big).
$$
Moreover, we have also $\OOO_{\T\setminus V(t)}(U)\cong\OOO_{\X^\beth_t}\big(\pi^{-1}(U)\cap \X_t^\beth\big)$, because, as noted in~\ref{bounded_functions_on_T}, an analytic function is bounded by $1$ at a point $x$ of $\X_t^\beth$ if and only if it is bounded by $1$ at all the points of the orbit $\I\cdot x$.
It remains to show that $\varphi$ is a homeomorphism of $G$-sites, i.e. that whenever $U$ is an analytic domain of $\X_t^\beth$ then $\varphi(U)$ is an analytic domain of $\T$. 
If $V$ is an affinoid domain of $\X_t^\beth$, then it is also an affinoid domain of $\Xeta$, hence the same reasoning as in the proof of Proposition~\ref{G-admissible_strip} shows that $\pi^{-1}\big(\varphi(U)\big)=\I\cdot U$ is an analytic domain of $\Xeta$, therefore $\varphi(U)$ is an analytic domain of $\T$.
\end{proof}

\begin{cor}\label{locally_analytic}
Let $\X$ be a special formal scheme over $k$. Then the normalized Berkovich space $\T$ of $\X$ is $G$-locally isomorphic in the category $\calC$ to an object in the image of $\forg:\big(An_{k((t))}\big)\to\calC$. 
\end{cor}

\begin{proof}
Without loss of generality we can suppose that $\X=\Spf(A)$ is affine. 
Choose generators $(f_1,\ldots,f_s)$ for an ideal of definition of $\X$.
Each $f_i$ is topologically nilpotent in $A$ and so induces a morphism $\X\to\Spf\big(k[[t]]\big)$.
Then the ${\X^\beth\setminus V({f_i})}$ cover $\Xeta$, hence $\T$ is covered by the ${\big(\X^\beth\setminus V({f_i})\big)}/(\I)$ and by~\ref{thm_structure_arboretum} ${\big(\X^\beth\setminus V({f_i})\big)}/(\I)\cong {\T\setminus V({f_i})}\cong \forg\big(\X^\beth_{f_i}\big)$.
\end{proof}

\Pa{Remarks}{
If $\X_t$ is a formal scheme of finite type over $k[[t]]$ then $t$ does not vanish on $\T$, and so once we have chosen a $t$-adic absolute value on $k((t))$ Theorem~\ref{thm_structure_arboretum} identifies $\T$ with the image in $\calC$ of the $k((t))$-analytic space $\X_t^\beth$.
Note that the local $k((t))$-analytic structures that we obtain in Corollary~\ref{locally_analytic} are far from being unique.
Indeed, not only we have to make choices of $|t|\in\left]0,1\right[$, but we have to choose an affine cover of $\X$ and generators of the ideals of definition of the elements of this covers.
Nonetheless, for the sake of simplicity we will ofter refer to this result by saying that normalized $k$-spaces are $G$-locally $k((t))$-analytic spaces.
}

\Pa{Example}{\label{analytic_structure_valuative_tree}
If $\X$ is the formal scheme $\Spf\big(\mathbb C[[X,Y]]\big)$, then its normalized Berkovich space $\T$ is the valuative tree of \cite{valtree}, as observed in Example~\ref{valtree}. 
The largest ideal of definition of $\mathbb C[[X,Y]]$ is $(X,Y)$, so $\T$ is the union of the two $k((t))$-analytic curves $\X^\beth_X$ and $\X^\beth_Y$, both isomorphic to the $1$-dimensional open analytic disc over $k((t))$. 
It's important to remark that on their intersection, which is $\T\setminus V(XY)$, the two $k((t))$-analytic structures do not agree, because $t$ is sent to $X$ for one of them and to $Y$ for the other.
Actually, more is true: we will show in Example~\ref{analytic_structure_valuative_tree_continued} that there is no $k((t))$-analytic space $C$ such that $\T\cong\forg(C)$.
Observe that the complement of $\X^\beth_X$ in $\T$, which is the zero locus of $X$, consists of exactly one endpoint of the valuative tree, namely the order of vanishing at the origin along $X$ (it is a \defi{curve valuation} in the terminology of \cite[1.5.5]{valtree}).
}

\pa{\label{dejong_original}
Corollary~\ref{locally_analytic} gives us a way of proving some assertions about analytic spaces over trivially valued fields by reducing to the non-trivially valued case. 
We prove in this way the analogue for normalized spaces of a result of A.J. de Jong \cite[7.3.6]{deJ95}. 
De Jong shows that if $\X=\Spf A$ is a normal and affine special formal scheme flat over a complete discrete valuation ring $R$, then the formal functions on $\X$ coincide with the analytic functions bounded by $1$ on $\X^{\beth}$. 
More precisely, this result holds under weaker assumptions: it is enough for $A$ to be $R$-flat, reduced, and integrally closed in the ring $A\otimes_R\mbox{Frac}(R)$ (this was remarked in \cite[7.4.2]{deJ95} and proven in \cite[2.1]{MartinKappen15}).
We can deduce that the same is true in our setting:
}

\begin{cor}\label{dejong_normalized}
Let $A$ be a special $k$-algebra and assume that $A$ is admissible and normal. 
If we denote by $\X$ the formal scheme $\Spf A$, then the canonical morphism $A\to\OOO_\T(\T)$ is an isomorphism.
\end{cor}

\begin{proof}
Since $A$ is admissible, it has a nonzero ideal of definition and $\T$ is non-empty.
By working separately on the connected components of $\X$ we can assume that it is connected and so $A$, being normal, is a domain.
Since $A$ is normal we have $A=\cap R$ for $R$ ranging among the valuation rings of rank one of the total ring of fractions of $A$, so the canonical morphism $A\to\OOO_\T(\T)$ is an inclusion.
If $t$ is a nonzero element of an ideal of definition of $A$, then $A$ is flat over $k[[t]]$ and, since $A$ is normal, it is integrally closed in $A[t^{-1}]$.
Then de Jong's theorem applies to $\X_t=\Spf{A}$, seen as a special formal scheme over $k[[t]]$, yielding $A\stackrel{\sim}{\to}\OOO_{\X_t^\beth}\big(\X_t^\beth\big)=\OOO_\T(U)$, where $U=\pi\big(\Xeta\setminus V(t)\big)$ is the subspace of $\T$ isomorphic to $\X^\beth_t$.
Since $t$ is not a zero divisor in $A$, then $V(t)$ is a thin subset of $\Xeta$ in the sense of \cite[\S3.3]{Ber90}, so the restriction $\OOO_\Xeta\left({\Xeta}\right)\to\OOO_\Xeta\big(\Xeta\setminus V(t)\big)$ is injective by \cite[3.3.14]{Ber90}.
From the chain of inclusions $A\hookrightarrow\OOO_\T(\T)\hookrightarrow\OOO_{T_{\X}}(U)=A$ we deduce that $A\cong\OOO_\T(\T)$.
\end{proof}

\pa{It $T$ is an object of $\calC$, we define the sheaf $\OOOO_T$ on $T$ as the subsheaf of $\OOO_T$ consisting of the sections that are non-invertible in every stalk. 
If $T=\T$ is the normalized space of a special formal scheme $\X$ over $k$ that is covered by the formal spectra of normal and admissible special $k$-algebras, then $\OOOO_T$ coincides with the sheaf $\OOOO_\T$ defined earlier. 
Moreover, in this case the largest ideal of definition of $\X$ coincides with $(\Sp_{\X})^*\OOOO_\T$. 
Indeed, the elements of the largest ideal of definition of $\X$ are precisely the ones that are topologically nilpotent, and this property can be verified by looking at the absolute values at every point.
}

\Pa{Example (continued)}{\label{analytic_structure_valuative_tree_continued}
We have discussed in \ref{analytic_structure_valuative_tree} a cover of the valuative tree $\T=T_{\Spf(\mathbb C[[X,Y]])}$ by $\mathbb C((t))$-analytic curves.
Using Corollary~\ref{dejong_normalized} we now show that there is no global $\mathbb C((t))$-analytic structure on $\T$.
To see this, assume that $\T\cong\forg(C)$ for some $\mathbb C((t))$-analytic space $C$, where $k((t))$ is endowed with the $t$-adic absolute value such that $|t|=r$. 
Then the image of $t$ in $\OO_C(C)\cong\OO_\T(\T)$, which by abuse of notation we still denote by $t$, has to be a nowhere vanishing function that is strictly bounded by 1 (because it takes the constant value $r<1$ on $C$), hence in particular an element of $\OOOO_\T(\T)$. 
Since $\OOOO_\T(\T)$ coincides with the largest ideal of definition of $\OOO_\T(\T)\cong$, which by Corollary~\ref{dejong_normalized} is isomorphic to $\mathbb C[[X,Y]]$, the element $t$ is a complex power series in $X$ and $Y$ with no constant term and therefore defines the germ of a curve at the origin of the affine plane $\mathbb A^2_\C$. 
Then the order of vanishing at the origin along this germ 
defines a point of $\T$ on which $t$ vanishes, giving a contradiction.
}

\begin{rem}
There are other examples of subspaces of analytic spaces that are naturally analytic spaces locally but do not have a canonical field of definition.
This is the case for the analytic boundaries of affinoid domains.
This kind of behavior appears for example in \cite[Lemme 3.1]{Ducros12}.
\end{rem}

We conclude the section by applying the results we have obtained to the study of admissible formal $k$-schemes. We start with an easy lemma.

\begin{lem}\label{lem_admissible_1}
Let $A$ be a special $k$-algebra. 
Then the morphism of formal schemes $\Spf(A/A_{\tors})\to\Spf A$ induced by the quotient $\pi:A\to A/A_{\tors}$ gives an isomorphism on the level of normalized Berkovich spaces.
\end{lem}

\begin{proof}
Let $\{g_1,\ldots,g_s\}$ be a set of generators of an ideal of definition of $A$ and denote by $\X$ and $\X'$ the formal spectra of $A$ and $A/A_{\tors}$ respectively. Then $\T$ is covered by the Berkovich spaces $\X^\beth_{g_i}$, the space $T_{\X'}$ is covered by the $(\X')^\beth_{g_i}$ and the morphism $f:T_{\X'}\to\T$ induced by $\pi$ is locally the morphism of Berkovich spaces induced by the morphism of special formal $k[[t]]$-schemes $(\X')_{g_i}\to\X_{g_i}$ coming from $\pi$. The latter is an isomorphism since every element of $A_{\tors}$ is $g_i$-torsion, hence $f$ is an isomorphism.
\end{proof}

We deduce the following result, which in turn implies that a formal $k[[t]]$-scheme of finite type is an admissible special formal $k$-scheme if and only if it is admissible in the classical sense.

\begin{prop}\label{admissible}
Let $\X=\Spf A$ be an affine special formal scheme over $k$. 
Then $\X$ is admissible if and only if the special $k$-algebra $A$ is admissible.
\end{prop}

\begin{proof}
To prove the ``if'' part, since the topology on $\X$ is generated by affine open formal subschemes, it is enough to show that the map $\varphi:A\to\OO_{\Xeta}(\Xeta)$ is injective whenever $A$ is admissible; \ref{admissible_equivalence} will then allow to conclude. 
So let $a$ be an element of $A\setminus\{0\}$ such that $\varphi(a)=0$. Choose $t$ in the largest ideal of definition of $A$ such that $a\notin A_{t-\tors}$ and consider the following commutative diagram:
\begin{displaymath}
    \xymatrix@R=1.5pc@M=3pt@L=3pt{
        A \ar[d]_\varphi \ar[r]^(.33){\pi} & A\otimes_{k[[t]]}k((t))  \ar@<-4ex>@{^(->}[d] \\
        \OO_{\Xeta}\big(\Xeta\big) \ar[r]       & \OO_{\X_t^\beth}\big(\X_t^\beth\big)}
\end{displaymath}
where the bottom map is the restriction map under the identification of $\X_t^\beth$ with the subspace $\pi\big(\Xeta\setminus V(t)\big)$ of $\T$. 
The right vertical arrow is injective, as shown in \ref{canonical_injection_special_algebras}.
Since $a\notin A_{t-\tors}$ then $a$ is sent to a nonzero element by the top map, hence also $\varphi(a)$ is different from zero, which is what we had to prove.
Now, to prove the ``only if'' part, denote by $\pi$ the quotient map $A\to A/A_{\tors}$, by $\X'$ the formal spectrum of $A/A_{\tors}$, and consider the following commutative diagram:
\begin{displaymath}
    \xymatrix@C=2.5pc@R=1.5pc@M=3pt@L=3pt{
        A \ar[d]_\varphi \ar[r]^(.4)\pi & A/A_{\tors}  \ar[d]_\psi\\
        \OO_{\Xeta}\big(\Xeta\big) \ar[r]^(.48)\simeq       & \OO_{\X'^*}\big(\X'^*\big) }
\end{displaymath}
where the bottom arrow is induced by the morphism induced on formal spectra by $\pi$; it is an isomorphism thanks to Lemma~\ref{lem_admissible_1}. 
Since $\psi$ is injective from the previous part, it follows that $\varphi$ is injective only if $\pi$ is injective, that is only if $A_{\tors}=0$.
\end{proof}

\pa{\label{associated_admissible}
Let $\X$ be a special formal scheme over $k$, and let $\mathscr{T}_\X$ be the subsheaf of $\OO_\X$ such that $\mathscr T_\X(\Spf A)=A_{\tors}$ for every affine subscheme $\Spf A$ of $\X$. 
It's a coherent ideal subsheaf of $\OO_\X$, so we can consider the quotient $\OO_\X/\mathscr T_\X$.
 The special formal scheme $\X_{\mathrm{adm}}$ is defined as the closed formal subscheme of $\X$ defined by $\mathscr T$. 
It is an admissible special formal scheme over $k$ that we call the \defi{admissible special formal scheme associated with} $\X$. 
Its normalized Berkovich space coincides with the one of $\X$.
}

The following proposition is the analogue of an important classical result for formal $R$-schemes of finite type.

\begin{prop}\label{sp_surjective}
Let $\X$ be an admissible special formal $k$-scheme. 
Then the maps $\Sp_\X\colon\Xeta\to\X$ and $\Sp_\X\colon\T\to\X$ are surjective.
\end{prop}

\begin{proof}
We can replace $\X$ by its formal blowup along $\X_0$, since the blowup morphism induces a surjective map between the reductions.
Indeed, without loss of generality we can assume that $\X=\Spf(A)$ is affine and since it is admissible then no component of $\Spec(A)$ is contained in $\X_0$.
Therefore the image of the scheme theoretic blowup of $\Spec(A)$ along $\X_0$, which is closed since the blowup is a proper map and has to contain the complement $\Spec(A)\setminus\X_0$, is all of $\Spec(A)$.
Note that the blown up formal scheme is still admissible, because its ideal of definition is locally principal, generated by a regular element.
By replacing $\X$ with an affine open formal subscheme whose ideal of definition is principal, generated by an element $t$, we can assume that $\X_t=\X=\Spf(A)$ is an affine formal scheme, flat and of finite type over $k[[t]]$.
Furthermore, we can assume that $A$ is integrally closed in $A[t^{-1}]$, since the morphism of formal schemes induced by taking the integral closure is surjective by \cite[V.2.1, Theorem 1]{Bourbaki-CommutativeAlgebra1-7}.
By \cite[2.1]{MartinKappen15} we have $A\cong\OOO_{\X_t^\beth}\big({\X_t^\beth}\big)$, therefore $(\X_t)_0=\X_0$ is the canonical reduction of the affinoid space $\X_t^\beth$, and the map $\Sp_{\X_t}\colon\X_t^\beth\to\X_0$ coincides with the reduction map of \cite[\S2.4]{Ber90}.
Therefore $\Sp_{\X_t}$ is surjective by \cite[2.4.4(i)]{Ber90}.
The surjectivity of $\Sp_\X\colon\Xeta\to\X$ follows from the fact that $\iota\circ\Sp_\X=\Sp_{\X_t}$, where $\iota\colon\X^\beth_t\to\Xeta$ is the natural inclusion of Lemma~\ref{lemma_omeo}.
\end{proof}


\section{Affinoid domains and atlases}
\label{section_3.3}

In this section we develop more thoroughly the analogy between normalized spaces over $k$ and analytic spaces over $k((t))$, by defining the class of affinoid domains of a normalized space and showing that they behave like the affinoid domains of analytic spaces. 
In particular, in Proposition~\ref{projection_of_affinoids} we show that the $G$-topology of a normalized space can be described in terms of its affinoid domains, and in Proposition~\ref{existence_atlas} we prove that normalized spaces are $G$-covered by finitely many affinoid domains.
Theorem~\ref{theorem_affinoids_intrinsic} shows that the property of being an affinoid domain of a normalized space is intrinsic, not depending on the choice of a $k((t))$-analytic structure.

\pa{
Let $V$ be an object of $\calC$. 
We say that $V$ is {\it affinoid} if it is isomorphic to $\forg(X)$ for some strictly affinoid $k((t))$-analytic space $X$.
A $G$-admissible subspace $V$ of an object $T$ of $\calC$ that is affinoid is said to be an {\it affinoid domain} of $T$.
}

\pa{
Equivalently, $V$ is affinoid if and only if it is isomorphic to the normalized space $\T$ of some affine formal $k[[t]]$-scheme of finite type $\X_t$, since in this case $\T\cong\forg\big(\X_t^\beth\big)$.
If we want to remember the element $t\in \OOO_V(V)$ that is the image of $t$ under the canonical homomorphism $k[[t]]\to \OOO_V(V)$, we say that $V$ is {\it affinoid with respect to the parameter} $t$, and by abuse of notation we will denote by $V_t$ both a strictly affinoid $k((t))$-analytic space whose image in $\calC$ is isomorphic to $V$ (a choice of $|t|\in\left]0,1\right[$ is implicit here) and the affinoid space $V$ itself.
Observe that the parameter $t$ is actually an element of $\OOOO_V(V)$.
}

\begin{rem}\label{remark_affinoid_not_canonical}
If $V$ is an object of $\calC$ that is affinoid with respect to two different parameters $t_1$ and $t_2$, it is not true in general that $V_{t_1}$ and $V_{t_2}$ are isomorphic as $k((t))$-analytic spaces. 
For example, the strictly affinoid $k((t))$-analytic spaces $\mathcal M\big({k((t))}\big)$ and $\mathcal M\big({k((t))\{X\}/(X^2-t)}\big)$ are not isomorphic since  $k((t))$ and $k((t))\{X\}/(X^2-t)$ are not isomorphic as $k((t))$-algebras, but the associated normalized spaces are isomorphic because $k((t))$ and $k((t))\{X\}/(X^2-t)\cong  k((X))$ are isomorphic as special $k$-algebras.
\end{rem}

The following proposition pushes further the analogies between usual Berkovich spaces and normalized spaces, showing that the $G$-topology of a normalized space can be described in terms of its affinoid domains.

\begin{prop}\label{projection_of_affinoids}
Let $\X$ be a special formal scheme over $k$ and let $U$ be an analytic domain of the normalized space $\T$ of $\X$. 
Then $U$ is $G$-covered by affinoid domains of $\T$.
\end{prop}

\begin{proof}
Since $\T$ is $G$-covered by the normalized spaces of the affine open formal subschemes of $\X$, we can assume that $\X$ is itself affine. 
Assume that $U$ is an analytic domain of $\T$, i.e. that $\pi^{-1}(U)$ is an analytic domain of $\Xeta$.
Cover $\T$ by the $k((t))$-analytic spaces $\X_{t_i}^\beth$, for $t_i$ ranging over a finite set of generators of an ideal of definition of $\X$, and set $U_{i}=\pi^{-1}(U)\cap\X_{t_i}^\beth$.
Now, each $U_{i}$ is an analytic domain of $\X_{t_i}^\beth$, so it is $G$-covered by affinoid domains $V_{i,j}$ of $\X_{t_i}^\beth$.
Therefore $\big\{\pi(V_{i,j})\big\}_j$ is a $G$-cover of $\pi(U_i)$ by affinoid domains of $\T$.
Since $\big\{\pi\big(\X_{t_i}^\beth\big)\big\}_i$ is a finite open cover of $\T$, the $\pi(U_{i})$ form a finite open cover of $U$, hence $\big\{\pi(V_{i,j})\big\}_{i,j}$ is a $G$-cover of $U$ by affinoid domains of $\T$.
\end{proof}

\begin{rem}
Proposition~\ref{projection_of_affinoids} tells us that we can think about the $G$-topology of $\T$ the same way we think about the one of a Berkovich space. 
If $U$ is a subset of $\T$, then $U$ is an analytic domain if and only if there exists a family $\{V_i\}_{i\in I}$ of affinoid domains of $\T$ contained in $U$ such that the following property holds: for every point $x$ of $U$, there exists a finite subset $I_x$ of $I$ such that $\bigcup_{i\in I_x}V_i$ contains an open neighborhood of $x$ and $x\in\bigcap_{i\in I_x}V_i$.
\end{rem}

\pa{
We showed in Corollary~\ref{locally_analytic} that the normalized Berkovich space $\T$ of a special formal $k$-scheme $\X$ is $G$-locally a $k((t))$-analytic space. 
We will now describe a second way of covering $\T$ by $k((t))$-analytic spaces that will be very useful later; the price to pay is that we are obliged to change the formal scheme $\X$.
If $T$ is an object of $\calC$, we define an \defi{atlas} of $T$ to be a $G$-cover of $T$ by affinoid domains.
}

\begin{prop}\label{existence_atlas}
Let $\X$ be a special formal $k$-scheme. 
Then the normalized Berkovich space $\T$ of $\X$ admits a finite atlas.
\end{prop}

\begin{proof}
By replacing $\X$ with the associated admissible special formal $k$-scheme, as defined in \ref{associated_admissible}, we can assume that $\X$ is admissible. 
Performing an admissible blowup $\X'\to\X$ we can assume that the largest ideal of definition of $\X$ is locally principal. 
Now, by locally sending $t$ to a generator of this ideal, we cover $\X$ by finitely many affine open formal schemes of finite type over $k[[t]]$. 
Their normalized spaces then form an atlas of $\T$.
\end{proof}

\begin{ex}
In the case of the valuative tree of Example~\ref{valtree}, we get an atlas by blowing up the origin of $\mathbb A^2_\C$ and using the two charts of the blowup.
\end{ex}


\pa{
If $X$ is a strict $k((t))$-affinoid space and $\X_t=\Spf(A)$ is a flat formal $k[[t]]$-scheme of finite type such that $\X_t^\beth\cong X$, then by flatness $A$ injects into $\calO_X(X)\cong A \otimes_{k[[t]]} k((t))$, and so $X$ is reduced if and only if $A$ is reduced.
%
%
}

\pa{
If $X$ is a reduced strict $k((t))$-affinoid space, then the algebra $\OOO_X(X)$ is topologically of finite type over $k[[t]]$ by the Grauert-Remmert finiteness theorem \cite[\S4, Endlichkeitssatz]{GrauertRemmert1966}, and we have isomorphisms of $k((t))$-analytic spaces $X \cong \mathcal M\big({\OOO_X(X)\hat\otimes_{k[[t]]}k((t))}\big) \cong \Spf\big(\OOO_X(X)\big)_t^\beth$.
If $\forg(X)$ is isomorphic to an object $V$ of $\calC$ then obviously $\OOO_X(X)\cong\OOO_V(V)$ as $k$-algebras.
Therefore, $V$ is reduced affinoid if and only if the $k$-algebra $\OOO_V(V)$ can be endowed with a structure of a reduced $k[[t]]$-algebra topologically of finite type and $V$ is isomorphic to 
$\forg\big(\Spf(\OOO_V(V))_t^\beth\big) \cong T_{\Spf(\OOO_V(V))}$ in $\calC$.
The special formal $k$-scheme $\Spf(\OOO_V(V))$ can then be seen as a \emph{canonical formal model} for $V$.
}


\pa{
Let $V$ be an affinoid object $\calC$ with respect to the parameter $t\in\OOOO_V(V)$.
We say that $V$ is \emph{distinguished} (with respect to the parameter $t$) if it is reduced and $\OOO_V(V)\otimes_{k[[t]]}k$ is reduced as well.
}

\pa{\label{distinguished_affinoid_equivalence}
Let $X=\calM(\calA)$ be a strictly affinoid $k((t))$-analytic space.
Then $\forg(X)$ is distinguished with respect to $t$ if and only if the affinoid algebra $\calA$ is distinguished in the classical sense, see \cite[\S6.4.3]{BGR}.
Indeed, $\calA$ is distinguished if and only if it is reduced and its spectral norm $|\cdot|_{\sup}$ takes values in $|k((t))|$ (this is \cite[6.4.3/1]{BGR}, since by \cite[3.6]{BGR} $k((t))$ is a stable field).
In particular, to show both implication we can assume that $X$ is reduced, and therefore its canonical model is $\X_t=\Spf\OOO_X(X)$.
Observe that $|\calA^\times|_{\sup}=|k((t))^\times|$ if and only if $t\OOO_X(X)=\big\{f\in\calA\,\big|\,|f|_{\sup}\leq|t|\big\}$ coincides with $\OOOO_X(X)=\big\{f\in\calA\,\big|\,|f|_{\sup}<1\big\}$.
It follows that if $\calA$ is distinguished then the special fiber $\X_{t,s}=\Spec\big(\OOO_X(X)/ t\OOO_X(X)\big)$ of $\X$ coincides with the usual reduction $\Spec\big(\OOO_X(X)/\OOOO_X(X)\big)$ of the affinoid $X$, which in particular proves that $\forg(X)$ is distinguished.
Conversely, assume that $|\calA|_{\sup}\supsetneq|k((t))|$, so that $\calA$ is not distinguished.
Then there exists an element $f\in \calA$ such that $|t|<|f|_{\sup}<1$.
This implies that there exists some $n>0$ such that $f^n\in t\OOO_X(X)$ while $f\in \OOO_X(X)\setminus t\OOO_X(X)$, hence $\X_{t,s}$ is not reduced, which proves that $\forg(X)$ is not distinguished.
This also shows that an affinoid object $V$ with respect to the parameter $t\in\OOOO_V(V)$ is distinguished with respect to $t$ if and only if $t$ is a generator of $\OOOO_V(V)$.
}

In the remaining part of the section we give a criterion for a reduced object of $\calC$ to be affinoid that does not require checking the existence of a parameter $t$, following the  approach of \cite{Liu90} for rigid analytic spaces.

\pa{\label{definition_pseudo-affinoid}
Let $X$ be a locally ringed $G$-topological space. 
Following \cite{Liu90}, we say that $X$ is a \defi{Stein space} if $\OO_X$ is coherent (considered as a sheaf of modules over itself) and we have $H^n(X,\mathscr F)=0$ for every coherent sheaf of $\OO_X$-modules $\mathscr F$ and every $n\geq1$.
\\
Let $T$ be an object of $\calC$. 
We say that $T$ is \defi{compact} if it is Hausdorff and every $G$-cover of $T$ has a finite $G$-subcover. 
Finally, we say that $T$ is \defi{pseudo-affinoid} if it is isomorphic to $\forg\big(\X_t^\beth\big)$ for some affine special formal scheme $\X_t$ over $k[[t]]$.
}

\begin{prop}\label{pseudo-affinoid_stein} Every pseudo-affinoid object of $\calC$ is a Stein space.
 If $X$ is a $k((t))$-analytic space, then $X$ is a $k((t))$-affinoid space if and only if $\forg(X)$ is both compact and pseudo-affinoid. 
\end{prop}

\begin{proof}
If an object $T$ of $\calC$ is pseudo-affinoid then as discussed in \ref{example_affine_beth_space} it is the increasing union of the affinoid domains $W_\varepsilon$, and we have surjective restriction morphisms $\OO_T(W_\varepsilon)\to\OO_T(W_{\varepsilon'})$ when $\varepsilon>\varepsilon'$, so $T$ is quasi-Stein in the sense of Kiehl. 
It follows that pseudo-affinoid objects are Stein spaces since Kiehl's Theorem B \cite[2.4]{Kie67} applies (see also \cite[\S2.1]{Nicaise09} for a definition of quasi-Stein and the statement of Kiehl's theorem).
To prove the second claim, note that a $k((t))$-affinoid space $X$ is compact, so $\forg(X)$ is compact.
The fact that it is pseudo-affinoid is standard: a model of $X$ is obtained by taking the formal spectrum of the image of $k[[t]]\{X_1,\ldots,X_n\}$ via an admissible epimorphism $k((t))\{X_1,\ldots,X_n\}\to\OO_X(X)$. 
Conversely, if $\forg(X)$ is pseudo-affinoid then the $W_\varepsilon$ above form a $G$-cover of it and so, since this family is increasing, by compactness $\forg(X)$ coincides with one of the $W_\varepsilon$, which is an affinoid domain with respect to the parameter $t$. 
Therefore $X$ is itself affinoid.
\end{proof}

\pa{\label{recall_liu}
Liu has proven in \cite[3.2]{Liu90} that if $X$ and $Y$ are two rigid spaces over a non-trivially valued non-archimedean field $K$ and $X$ is Stein and quasi-compact, then the canonical map 
\[
\Hom_{K\mathrm{-an}}(Y,X)\to\Hom_{K\mathrm{-alg}}\big(\OO_X(X),\OO_Y(Y)\big)
\]
 is a bijection.
He deduced that a rigid space $X$ over $K$ is strictly affinoid if and only if it is Stein, $\OO_X(X)$ is a strict affinoid $K$-algebra and $X$ is quasi-compact \cite[3.2.1]{Liu90}.
We will prove a similar result for the reduced objects of the category $\calC$.
}

\begin{thm}\label{theorem_affinoids_intrinsic}
Let $X$ be a reduced $k((t))$-analytic space. 
Then $X$ is strictly $k((t))$-affinoid if and only if $\forg(X)$ is Stein and compact and $\OOO_X(X)$, with its $\OOOO_X(X)$-adic topology, is a special $k$-algebra.
\end{thm}

\begin{proof}
If $A$ is a reduced strictly affinoid $k((t))$-algebra then by \cite[\S4, Endlichkeitssatz]{GrauertRemmert1966} $A^\circ$ is a $k[[t]]$-algebra topologically of finite type. 
In particular it is special over $k$, so the ``only if'' implication is clear. 
For the converse implication, define $\X=\Spf\big(\OOO_X(X)\big)$. 
Since $X$ is a $k((t))$-analytic space, the image of $t$ in $\OO_X(X)$ is strictly bounded by $1$, hence it's an element of the largest ideal of definition of $\OO_\X(\X)$,
so $\X_t$ is an affine special formal scheme over $k[[t]]$. 
It follows that $\X_t^\beth$ is a pseudo-affinoid space.
Since $X$ is compact and $|t|<1$, we have $\OOO_X(X)[t^{-1}]\cong\OO_X(X)$. 
Indeed, any $f$ in $\OO_X(X)$ is bounded on $X$, hence $ft^n$ is bounded by $1$ for $n$ big enough. 
Therefore, by \cite[3.2]{Liu90}, as recalled in~\ref{recall_liu}, the canonical homomorphism of $k((t))$-algebras $\OOO_X(X)[t^{-1}]\to\OO_{\X_t^\beth}\big(\X_t^\beth\big)$ induces a morphism of rigid $k((t))$-analytic spaces $\X_t^{\mathrm{rig}}\to X^{\mathrm{rig}}$.
Observe that we have $X^{\mathrm{rig}}=\mathrm{Specmax}\big(\OO_X(X)\big)$ by \cite[1.3]{Liu90}, and similarly $\X_t^{\mathrm{rig}}=\mathrm{Specmax}\big(\OO_{\X_t^\beth}\big(\X_t^\beth\big)\big)$ by \cite[7.1.9]{deJ95}.
On points, the morphism $\X_t^{\mathrm{rig}}\to X^{\mathrm{rig}}$ is then obtained by sending a maximal ideal of $\OO_{\X_t^\beth}\big(\X_t^\beth\big)$ to the inverse image of this ideal under the composition $\OO_X(X)\to\OO_{\X_t^\beth}\big(\X_t^\beth\big)$.
It follows from \cite[7.1.9]{deJ95} and \cite[1.3]{Liu90} that this morphism is a bijection and that moreover it induces isomorphisms at the level of completed local rings.
Since $X^{\mathrm{rig}}$ is quasi-compact, it follows from \cite[\S7.3.3 Proposition 5]{BGR} that $X^{\mathrm{rig}}\to\X_t^{\mathrm{rig}}$ is an isomorphism of rigid spaces.
Therefore by \cite{Ber93} the morphism of analytic spaces $X\to\X_t^{\beth}$ is an isomorphism as well.
Hence, being both pseudo-affinoid and compact, $X$ is affinoid by Proposition~\ref{pseudo-affinoid_stein}.
Since $\OOO_X(X)$ is a special $k$-algebra, $X$ is moreover strict.
\end{proof}
A $k((t))$-analytic space $X$ is Stein if and only if $\forg(X)$ is, because the groups $\mathrm H^n(X,\mathscr F)$ depend only on the locally ringed site $(X,\OO_X)$ and not on the $k((t))$-algebra structure on $\OO_X$. 
Since the same is true for the two other properties in the statement above, we obtain the following corollary. 

\begin{cor}\label{affinoids}
Let $X$ and $X'$ be reduced $k((t))$-analytic spaces such that $\forg(X)\cong \forg(X')$. 
Then $X$ is strictly affinoid if and only if $X'$ is strictly affinoid.
\end{cor}


If $V$ is a subspace of an object of $\calC$ of the form $\forg(X)$ for some $k((t))$-analytic space $X$, then $V$ inherits the a structure of a $k((t))$-analytic space. 
Therefore, the last corollary has the following useful consequence.

\begin{cor}\label{affinoids2}
Let $X$ be a $k((t))$-analytic space and let $V$ be a reduced analytic domain of $\forg(X)$. 
Then $V$ is an affinoid domain of $\forg(X)$ if and only if it is of the form $\forg(W)$ for some strict affinoid domain $W$ of $X$. 
Moreover, $V$ is distinguished with respect to $t|_V$ if and only if $W$ is a distinguished strictly affinoid $k((t))$-analytic space.
\end{cor}


\section{Functoriality}
\label{section_3.4}

In this section we introduce the category of normalized spaces over $k$. 
The main result, Theorem~\ref{thm_functoriality}, is the analogue for normalized spaces of the fundamental result of Raynaud  (\cite{Raynaud74}, a detailed proof is in \cite[4.1]{BosLut93}) that states that the functor $\X\mapsto\X^\beth$ induces an equivalence between the category of admissible formal schemes of finite type over a complete valuation ring of height one $R$, localized by the class of admissible blowups, and the category of compact and quasi-separated strictly analytic spaces over $\mathrm{Frac}(R)$.

\pa{\label{definition_normalized_spaces}
We say that an object $T$ of $\calC$ is a \defi{normalized space} if it is compact and it has an atlas.
Observe that such a $T$ is then \emph{quasi-separated}, which means that the intersection of any two affinoid domains of $T$, being a compact analytic domain of each of them, is a finite union of affinoid domains of $T$.
A morphism $f\in \mbox{Hom}_{\calC}(T',T)$ is said to be a \defi{morphism of normalized spaces} if there exist finite covers $\{V_i\}_{i\in I}$ of $T$ and $\{W_j\}_{j\in J}$ of $T'$ by affinoid domains such that for every $i$ in $I$ there is a subset $J_i$ of $J$ with $f^{-1}(V_i)=\cup_{j\in J_i}W_j$ and, for every $j$ in $J_i$, the restriction $f|_{W_j}:W_j\to V_i$ of $f$ to $W_j$ is induced by a morphism of $k((t))$-analytic spaces. 
This means that $f|_{W_j}$ is associated with a $k((t))$-homomorphism $\calO_T(V_i)\to\calO_{T'}(W_j)$, where $\calO_T(V_i)$ and $\calO_{T'}(W_j)$ are seen as $k((t))$-affinoid algebras.
A more intrinsic way to see morphisms of normalized spaces is as those morphisms in $\calC$ that are induced by adic morphisms between formal models, as will be clear from Theorem~\ref{thm_functoriality}.
The \defi{category of normalized spaces} $(NAn_k)$ is the subcategory of $\calC$ whose objects are normalized spaces and whose morphisms are morphisms of normalized spaces.
}


\pa{
We want to prove an analogue of Raynaud's theorem for the functor $T\colon \X\mapsto T_\X$ going from the category of admissible special formal $k$-schemes with adic morphisms, which we denote by $(SFor_k)$, to $(NAn_k)$.
We will show that the functor $T$ is the localization of the category $(SFor_k)$ by the class $B$ of admissible formal blowups and we will characterize its essential image. 
The fact that $T$ sends admissible blowups to isomorphisms is~\ref{invariance_generic_fiber}.
}

\pa{
The category $(SFor_k)$ admits calculus of (right) fractions with respect to the class of morphisms $B$, in the sense of \cite[Ch. I]{GabrielZisman67}. 
This follows easily from the universal property of blowups and the results of \ref{basic properties blowup}.
%
Therefore, the localized category $(SFor_k)_B$ can be described in a simple way: its objects are the objects of $(SFor_k)$, and a morphism $\Y\to\X$ is a two-step zigzag 
$$
\xymatrix@R=.3pc@M=3pt@L=3pt{
& \Y'  \ar[dl]_w\ar[dr]^f\\
\Y  && \X 
}$$
where $f$ is a morphism in $(SFor_k)$ and $w$ is an admissible blowup, modulo the equivalence relation given by further blowing up $\Y'$. 
Such a morphism can be thought of as a fraction $fw^{-1}$.
Moreover, the localization functor $(SFor_k)\to(SFor_k)_B$ is left exact, and therefore preserves finite limits.
}

\begin{thm}\label{thm_functoriality}
 The functor $T:\X\mapsto\T$ induces an equivalence between the category $(SFor_k)_B$, the localization of the category $(SFor_k)$ of admissible special formal $k$-schemes with adic morphisms by the class $B$ of admissible blowups, and the category $(NAn_k)$.
\end{thm}

The remaining part of this section will be devoted to the proof of this result.

\begin{lem}\label{L: model of affinoid}
Let $\T$ be the normalized space of an admissible special formal $k$-scheme $\X$. 
Then:
\begin{enumerate} 
	\item If $V$ is an affinoid domain of $\T$, then there exist an admissible formal blowup $\X'\to\X$, a formal $k[[t]]$-scheme of finite type $\calV$ and an open immersion of formal $k$-schemes $\calV\hookrightarrow\X'$ inducing an isomorphism $T_{\calV}\cong V$ in $\calC$. 
	\item If $\{V_j\}_{j\in J}$ is a finite atlas of $\T$, then there exist an admissible formal blowup $\X'\to\X$ and a cover $\{\calV_j\}_{j\in J}$ of $\X$ by open formal subschemes such that $T_{\calV_j}\cong V_j$ for every $j$.
\end{enumerate}
\end{lem}

\begin{proof} We prove $(i)$ by reducing to the classical case of formal $k[[t]]$-schemes of finite type, where it is \cite[8.4.5]{Bosch14}. 
After an admissible blowup we can assume that $\X$ is covered by open formal $k[[t]]$-schemes of finite type $\X_i$, as in the proof of Proposition~\ref{existence_atlas}. 
Then $\T=\bigcup_i T_{\X_i}$, and so $V=\bigcup_i (V\cap T_{\X_i})$. 
Inducting on $i$, \cite[8.4.5]{Bosch14} tells us that after an admissible blowup $\X'_i$ of $\X_i$, which extends to an admissible blowup $\X'$ of $\X$, we can find an open formal $k[[t]]$-subscheme $\calV_i$ of $\X'_i$, which is also an open formal $k$-subscheme of $\X'$, such that $T_{\calV_i}=V_i$. 
We obtain an open formal $k$-subscheme $\calV=\bigcup_i\calV_i$ of an admissible blowup of $\X$ such that $T_\calV=V$, which is what we wanted.
To prove $(ii)$, we apply $(i)$ to get for every $j$ an admissible blowup $\X_j$ of $\X$ with an open formal subscheme $\calW_j\subset\X_j$ such that $T_{\calW_j}\cong V_j$. 
Using \ref{blowup: domination} we take an admissible blowup $\X'$ of $\X$ dominating all $\X_j$ via maps $f_j:\X'\to\X_j$, and we set $\calV_j\coloneqq f_j^{-1}(\calW_j)$. 
The $\calV_j$ form a cover of $\X'$ that satisfies the requirements since $T_{\calV_j}\cong T_{\calW_j}\cong V_j$.
\end{proof}

The rest of the proof of Theorem~\ref{thm_functoriality} will be divided in six steps. 
While the result could be proven in an analogous way as Raynaud did, we decided to deduce it from his result, to give an idea of how to apply to normalized spaces standard techniques over $k[[t]]$, in a similar way as what we did in Lemma~\ref{L: model of affinoid}.

\begin{step} {\it The functor $T$ factors through $(NAn_k)$.}
If $\X$ is a special formal scheme over $k$, then $T_\X$ is a normalized space, since it admits a finite atlas by Proposition~\ref{existence_atlas}. 
Now let $f:\Y\to\X$ be an adic morphism of special formal $k$-schemes. 
By replacing $\X$ and $\Y$ by admissible blowups we obtain an adic morphism $f':\Y'\to\X'$ and we can assume that $\X'$ is covered by affine open formal subschemes $\X_i$ of finite type over $k[[t]]$. 
The open formal $k$-subschemes $(f')^{-1}(\X_i)$ of $\Y'$ are themselves covered by affine open formal $k$-subschemes $\Y_{i,j}$, and since $f'$ is adic, the morphisms of formal $k$-schemes $f'|_{\Y_{i,j}}:\Y_{i,j}\to\X_i$ can be upgraded to a morphism of formal $k[[t]]$-schemes, by choosing the $k[[t]]$-structure on $\Y_{i,j}$ given by the parameter $f'_\#t$, and therefore it induces a morphism of $k((t))$-analytic spaces $(\Y_{i,j})_t^\beth\to(\X_i)_t^\beth$. 
This shows that $T_f$ is a morphism in $(NAn_k)$.
\end{step}

\begin{step} {\it Faithfulness.} Let $f,g:\Y\to\X$ be two morphisms in $(SFor_k)$ such that the induced morphisms of normalized spaces $f_T,g_T:T_\Y\to\T$ coincide. 
Since $\X$ and $\Y$ are admissible, the specialization maps $T_\X\to\X$ and $T_\Y\to\Y$ are surjective by Proposition~\ref{sp_surjective}.
Consider the following diagram:
\begin{displaymath}
    \xymatrix@C=3.5pc@R=1.5pc@M=3pt@L=3pt{
        T_\Y \ar@{>>}[d]_{\Sp_\Y} \ar[r]^{f_T=g_T} & T_\X  \ar@{>>}[d]^-{\Sp_{\X}} \\
        \Y \ar@<.7ex>[r]^f\ar@<-.7ex>[r]_g       & \X }
\end{displaymath}
\noindent It commutes for both choices of the map on bottom, so it follows that $f$ and $g$ coincide as maps between the topological spaces underlying $\Y$ and $\X$. 
Therefore we can assume that $\X$ and $\Y$ are affine, $\X=\Spf(A)$ and $\Y=\Spf(B)$, so that $f$ and $g$ correspond to two k-algebra maps $f_\sharp$ and $g_\sharp:A\to B$ respectively. 
Consider then the following diagram
\begin{displaymath}
\xymatrix@R=1.5pc@C=3.5pc@M=4pt@L=3pt{
        A \ar@{^{(}->}[d] \ar@<.7ex>[r]^{f_\sharp}\ar@<-.7ex>[r]_{g_\sharp} & B  \ar@{^{(}->}[d] \\
        \OO(\Xeta)    \ar[r]_{(f^*)_\sharp=(g^*)_\sharp}    & \OO(\Y^*) }
\end{displaymath}
that commutes for both choices of the map on top.
The vertical arrows are injective since $\X$ and $\Y$ are admissible, so $f_\sharp=g_\sharp$, and hence $f=g$.
\end{step}

\begin{step} {\it Fullness (modulo admissible blowup).} 
Let $f:T'\to T$ be a morphism in $(NAn_k)$ and let $\X$ and $\mathscr Y$ be models of $T$ and $T'$ respectively. 
Given two finite affine covers of $\X$ and of $\Y$, using the fact that admissible blowups open formal subschemes can be extended by~\ref{blowup: extension}, after blowing up $\Y$ we can refine them to finite covers $\{(\mathscr V_i)_t\}$ of $\X$ and $\{(\mathscr W_j)_t\}$ of $\Y$ by affine formal schemes of finite type over $k[[t]]$ in such a way that, if we define $k((t))$-analytic spaces $V_i=(\mathscr V_i)_t^\beth$ and $W_j=(\mathscr W_j)_t^\beth$, then $\{\forg(V_i)\}$ and $\{\forg(W_i)\}$ are covers of $T$ and $T'$ respectively as the ones in the definition of a morphism of normalized spaces given in \ref{definition_normalized_spaces}; in particular for every $j$ there exist an $i$ and a morphism of $k((t))$-analytic spaces $W_j\to V_i$ that lifts $f|_{\forg(W_j)}$.
For every $j$ we use Raynaud's theorem (see in particular assertion (c) after Theorem 4.1 in \cite{BosLut93}): after blowing up $(\mathscr W_j)_t$ to $(\mathscr W_j')_t$, the $k((t))$-analytic morphism $W_j\to V_i$ lifts to a morphism $F:(\mathscr W_j')_t\to(\mathscr V_i)_t$ of formal schemes of finite type over $k[[t]]$.
These morphisms glue to a morphism $F:\Y'\to\X$ of formal schemes over $k$ from a blowup $\Y'$ of $\Y$ to $\X$, and $T(F)=f$ since this is the case locally. 
The morphism $F$ is adic since it is locally a morphism of $k[[t]]$-formal schemes of finite type and such morphisms are always adic.
\end{step}

\begin{step}\label{proof raynaud isomorphism} {\it Isomorphisms come from admissible blowups.} 
If in the previous step we take for $f$ an isomorphism in $\calC$, then it can be lifted to an admissible blowup $F:\Y'\to\X$.
Indeed, if $f$ is an isomorphism, then the analytic morphisms $f|_{W_j}:W_j\to V_i$ of the previous step are all immersions of an affinoid domain in a $k((t))$-analytic space, so we can use the analogous result in Raynaud theory.
\end{step}

\begin{step} {\it Existence of a model.} 
Let $T$ be an element of $(NAn_k)$ and let $\{X_i\}_{i\in I}$ be a finite atlas of $T$. 
We will prove the existence of a model of $T$, i.e. a special formal $k$-scheme $\X$ such that $\T\cong T$, by induction on the cardinality of $I$. 
If $I$ consists of only one element then $T$ is affinoid, and therefore has a model. 
If $I=\{1,2,\ldots,n\}$, then $X_n$ being affinoid has a model $\calU$, and by induction we can find a model $\calV$ of $V\coloneqq X_1\cup\cdots\cup X_{n-1}$. 
Set $W=V\cap X_n$. 
Since $T$ is quasi-separated, $W$ admits a finite cover by affinoid domains, and this cover can be enlarged to a cover of $V$ by affinoid domains. 
By Lemma~\ref{L: model of affinoid}, there exist an admissible blowup $\calV'\to\calV$ and an open immerson $\calW_1\hookrightarrow\calV'$ inducing an isomorphism $T_{\calW_1}\cong W$. 
Similarly, there exist an admissible blowup $\calU'\to\calU$ and an open immerson $\calW_2\hookrightarrow\calU'$ inducing an isomorphism $T_{\calW_2}\cong W$. 
Since $\calW_1$ and $\calW_2$ are both models of $W$, using {\it Step \ref{proof raynaud isomorphism}} we can find an admissible blowup $\calW$ of both $\calW_1$ and $\calW_2$. 
These blowups can be extended using~\ref{blowup: extension} to admissible blowups $\calV''\to\calV$ and $\calU''\to\calU'$, so we can glue $\calV''$ and $\calU''$ along $\calW$, obtaining a model of $T$.
\end{step}

\begin{step} {\it End of proof.} It remains to prove that the functor $T$ satisfies the universal property of the localization of categories. 
The fact that $T$ sends admissible blowups to isomorphisms is~\ref{invariance_generic_fiber}. 
Given a category $\calC$ and a functor $F:(SFor_k)\to\calC$ such that $F(b)$ is an isomorphism in $\calC$ for every admissible blowup $b$, we need to show that $F$ factors as $G\circ T$ for a unique functor $G:(NAn_k)\to\calC$.
This is done in the exact same way as in the proof of Raynaud's theorem, using the basic results on blowups of~\ref{basic properties blowup}.
\end{step}

\noindent This completes the proof of Theorem~\ref{thm_functoriality}.\hfill\qedsymbol

\begin{rem}
In particular, an analytic space of the form $\Xeta$ for some special formal $k$-scheme $\X$ is uniquely determined by the associated normalized space. 
An explicit way of retrieving the topological space underlying $\Xeta$ from $T_\X$ is the following. 
If we cover $\T$ by affinoid subspaces $X_i$ with respect to parameters $t_i$, we need to glue the topological spaces $X_i\times\I$ along subspaces homeomorphic to $X_{ij}\times\I$. 
The gluing data is encoded in the $t_i$: if $x$ is a point of $X_{ij}$, we identify $(x,\gamma)\in X_i\times\I$ to $(x,\lambda_{ij}(x)\gamma)\in X_j\times\I$, where $\lambda_{ij}(x)\coloneqq \log|t_j|/\log|t_i|(x)$ is defined as in Remark~\ref{functions_on_T}.
\end{rem}


\section{Modifications of surfaces and vertex sets}
\label{section_4.1}

Starting from this section we move to the study of pairs $(X,Z)$, where $X$ is a surface over an arbitrary (trivially valued) field $k$ and $Z$ is a closed subvariety of $X$ containing its singular locus.
After giving some general definitions, in Theorem~\ref{vertex_sets} we use normalized spaces to produce formal modifications of $(X,Z)$ with prescribed exceptional divisors.

\pa{\label{setting_surfaces}
 Let $X$ be a surface over $k$, that is a geometrically integral and generically smooth $k$-scheme of dimension 2, and let $Z\subsetneq X$ be a nonempty closed subscheme whose support contains the singular locus of $X$. 
We denote by $\X$ the formal completion of $X$ along $Z$, and by $\Txz$ the normalized space $\T$ of $\X$.
We also call $\Txz$ the \defi{normalized space of the pair $(X,Z)$}. 
As discussed in Example~\ref{example_explicit_normalization}, $\Txz$ can be viewed as a non-archimedean model for the link of $Z$ in $X$.
We denote by $\widetilde{\Bl_{Z}X}$ the normalization of the blowup of $X$ along $Z$, and we write $\widetilde \X$ for the formal completion of $\widetilde{\Bl_{Z}X}$ along $\widetilde{\Bl_{Z}X}\times_{X}Z$. 
Then $\widetilde \X$ is another special formal $k$-scheme that is a model of $\Txz$.
By~\cite[2.16.5]{Nicaise09}, $\widetilde{\X}$ is isomorphic to the normalization of the formal blowup of $\X$ along $Z$.
}

\pa{ 
A \defi{log modification} of the pair $(X,Z)$ is a pair $(Y,D)$ consisting of a normal $k$-variety $Y$ and a Cartier divisor $D$ of $Y$, together with a proper morphism of $k$-varieties $f:Y\to X$ such that $D=Y\times_XZ$ as subschemes of $Y$ and $f$ is an isomorphism outside of $D$.
A log modification $(Y,D)$ of $(X,Z)$ is said to be a \defi{log resolution} of $(X,Z)$ if $Y$ is regular and $D$ has normal crossings (by which we mean that $D_{\mathrm{red}}$ has normal crossings, but not necessarily strict normal crossings, in the usual sense).
Note that $D$ being Cartier is not equivalent to the set-theoretic inverse image $f^{-1}(Z)=D_{\mathrm{red}}$ being Cartier, therefore our notion of log resolution is different from the notion of good resolution that is sometimes found in the literature.
}

\pa{
A \defi{formal log modification} of the pair $(X,Z)$ is a normal special formal $k$-scheme $\Y$ together with an adic morphism $f:\Y\to\X$ that induces an isomorphism of normalized spaces $T_\Y\stackrel{\sim}{\longrightarrow}\Txz$, and such that $\Y\times_\X Z$ is a Cartier divisor of $\Y$.
 If moreover $\Y$ is regular and $\Y\times_\X Z$ has normal crossings in $\Y$, then $\Y$ is said to be a \defi{formal log resolution} of $(X,Z)$.
}

\begin{lem}\label{lemma_properness_formal_modifications}
Let $f:\Y\to\X$ be a formal log modification of $(X,Z)$. 
Then $f$ is proper.
\end{lem}

\begin{proof}
Since $f$ is adic by definition, it is enough to show that the induced morphism $f_0:\Y_0\to\X_0$ is a proper morphism of schemes. 
Consider the morphism $\alpha:T_\X\to T_\Y$, inverse of the isomorphism induced by $f$. 
Theorem~\ref{thm_functoriality} states that there exists an admissible blowup $\tau:\X'\to\X$ and a morphism $g:\X'\to\Y$ such that $\alpha$ is induced by $g$. 
Since the composition $\X'\stackrel{g}{\longrightarrow}\Y\stackrel{f}{\longrightarrow}\X$ is $\tau$, the induced map $f_0\circ g_0\colon\X'_0\to\X_0$ is proper.
The map $g_0$ is surjective because $\Sp_{\Y}$ is surjective and the following diagram
\begin{displaymath}
    \xymatrix@C=2.5pc@R=1.5pc@M=3pt@L=3pt{
        T_{\X'} \ar[d]_{\Sp_{\X'}} \ar[r]^\simeq & T_\Y  \ar[d]^-{\Sp_{\Y}} \\
        \X'_0 \ar[r]_{g_0}       & \Y_0 }
\end{displaymath}
is commutative.
Since surjectivity is stable under base change by~\cite[3.5.2]{EGA1}, it follows that $f_0$ is universally closed and therefore proper.
\end{proof}

\pa{
If $(Y,D)\to(X,Z)$ is a log modification, then the formal completion $\Y=\widehat{Y/D}\to\X$ of $Y$ along $D$ is a formal log modification of $(X,Z)$.
Such a formal log modification $\Y$ of $(X,Z)$ is said to be \defi{algebraizable}, and a log modification $(Y,D)$ of $(X,Z)$ such that $\Y\to\X$ is isomorphic to $\widehat{Y/D}\to\X$ is called an \defi{algebraization} of $\Y$.
By Grothendieck's formal GAGA theorem \cite[5.1.4]{EGA3.1}, a log modification $(Y,D)\to(X,Z)$ is uniquely determined by the formal log modification $\widehat{Y/D}\to\X$ it algebraizes, and if $\Y\cong\widehat{Y/D}$ and $\Y'\cong\widehat{Y'/D'}$ are two formal log modifications that are algebraizable then $\Hom_\X(\Y',\Y)\cong\Hom_X(Y',Y)$.
If $\Y$ is a formal log modification of $(X,Z)$ algebraized by the log modification $(Y,D)$, since both the properties of being regular and of having normal crossings are local properties and, by excellence, can be checked on completed local rings, then $\Y$ is a formal log resolution of $(X,Z)$ if and only if $(Y,D)$ is a log resolution of $(X,Z)$.
In the following proposition we will prove that every formal log resolution if algebraized by a log resolution.
Since the normal crossing condition on the exceptional divisor is not needed, we give a slightly more general statement.
}

\begin{prop}\label{proposition_algebraization_resolutions}
Let $\Y$ be a formal log modification of $(X,Z)$ and assume that $\Y$ is regular.
Then $\Y$ is algebraizable.
If moreover $\Y$ is a formal log resolution, then it is algebraizable by a log resolution of $(X,Z)$.
\end{prop}

\begin{proof}
Let $f\colon\Y\to\X$ be a formal log modification of $(X,Z)$ and assume that $\Y$ is regular.
The map $f$ factors through a morphism $g:\Y\to\widetilde{\X}$.
It follows from Lemma~\ref{lemma_properness_formal_modifications} that $g$ makes $\Y$ into a proper, adic formal $\widetilde{\X}$-scheme.
Since $\widetilde\X$ is normal, $g$ is an isomorphism outside of the inverse image of a finite set of closed points of $\widetilde\X$.
Let $\mathscr U$ be an open and affine formal subscheme of $\widetilde \X$ such that there is exactly one point $x$ in $\mathscr U$ such that $g|_{g^{-1}(\mathscr U)}$ is an isomorphism outside of $g^{-1}(x)$, and denote by $E_1,\ldots,E_r$ the irreducible components of $g^{-1}(x)$.
Since $g^{-1}(\mathscr U)\subset\Y$ is regular, each $E_i$ is a Cartier divisor on $g^{-1}(\mathscr U)$ (the theory of Cartier divisors is developed over any ringed space, see for example \cite[\S21]{EGA4.4}). 
The intersection matrix $(E_i\cdot E_j)_{1\leq i,j\leq r}$ is negative definite because the whole of $g^{-1}(x)$ gets contracted to $x$ by $g$, so by the elementary (albeit long) linear algebra computation in \cite[page 138, $(ii)$]{Lipman69} we can find integers $a_i\leq0$ such that if we set $E=\sum_i a_iE_i$ we have $E\cdot E_i<0$ for every $i=1,\ldots, r$.
Consider the invertible sheaf $\calL=\OO_{g^{-1}(\mathscr U)}(E)\subset\OO_{g^{-1}(\mathscr U)}$ on ${g^{-1}(\mathscr U)}$ associated with $E$, and denote by $\calL_0$ the base change of $\calL$ to ${g^{-1}(\mathscr U)}\times_{\X}Z$.
Let $y$ be a point of $\widetilde \X_0$. 
If $y\neq x$ then $g_0^{-1}(y)$ is a point and $\calL_0|_{g_0^{-1}(y)}$ is therefore ample.
On the other hand, if $y=x$ then $\calL_0|_{g_0^{-1}(y)}$ is ample by Kleiman's criterion \cite[\S III.1]{Kleiman66} because of the inequalities $E\cdot E_i<0$.
Since $g_0$ is proper, by~\cite[4.7.1]{EGA3.1} this implies that the invertible sheaf $\calL_0$ is relatively ample with respect to $g|_{g^{-1}(\mathscr U)}$, and therefore it is ample since $\mathscr U$ is affine.
Since $\mathscr U$ is algebraized by an open subscheme $U$ of $\widetilde{\Bl_{Z}X}$, Grothendieck's existence theorem~\cite[5.4.5]{EGA3.1} guarantees that $g^{-1}(\mathscr U)$ is algebraized by a proper $U$-scheme.
Since those algebraizations are unique, we can glue them and so we deduce that $\Y$ is algebraized by a $k$-scheme $Y$, endowed with a proper morphism $g:Y\to\widetilde{\Bl_{Z}X}$.
Set $D=Y\times_X Z$; then $D$ is Cartier in $Y$ by the universal property of ${\Bl_{Z}X}$.
Since $f\colon\Y\to\X$, hence the morphism $\Y\to\widetilde{\X}$, induces an isomorphism at the level of normalized spaces, by Theorem~\ref{thm_functoriality} it is an isomorphism modulo admissible blowups, therefore its algebraization $g$ induces an isomorphism outside of $D$, and so $(Y,D)$ is a log modification of $(X,Z)$ algebraizing $\Y$.
\end{proof}

\pa{If $f\colon\Y\to\X$ and $f'\colon\Y'\to\X$ are two formal log modifications of $(X,Z)$, we say that $\Y'$ \defi{dominates} $\Y$ if there is a morphism of formal schemes $g\colon\Y'\to \Y$ such that $f\circ g=f'$; we denote this by $\Y'\geq \Y$. 
Note that if such a morphism $g$ exists, then it is unique.
This follows from the fact that $f$ is uniquely determined by the fact that it induces an isomorphism of normalized spaces: to prove this we can assume that $\X$ and $\Y$ are both affine, and conclude by observing that the image of an element of $\calO_\X(\X)$ in $\calO_\Y(\Y)$ only depends on its image in $\calO_{T_\Y}(T_\Y)$ because the natural map $\calO_\Y(\Y)\to\calO_{T_\Y}(T_\Y)$ is injective.
Two formal log modifications  $\Y'$ and $\Y$ are \defi{isomorphic} if $\Y \geq \Y'\geq \Y$, i.e. if there is an isomorphism $\Y' \cong \Y$ commuting with the morphisms to $\X$.
The domination relation is a filtered partial order on the set of isomorphism classes of formal log modifications of $X$. 
By the universal properties of blowup and normalization, this partially ordered set has $\widetilde{\X}$ as unique smallest element.
}

\pa{
If $\Y$ is a formal log modification of $(X,Z)$, we denote by $\Div_{X,Z}(\Y)$ the finite non-empty subset of $\Txz$ consisting of the $\I$-orbits of the divisorial valuations associated with the components of $\Y\times_\X Z$. 
If $\Y$ is algebraized by a log modification $(Y,D)$, we will also denote $\Div_{X,Z}(\Y)$ by $\Div_{X,Z}(Y)$.
We write $\Div_{X,Z}$ for the union of the sets $\Div_{X,Z}(\Y)$, for $\Y$ ranging over all the formal log modifications of $(X,Z)$; it is the set of the $\I$-orbits of the divisorial valuations on $X$ whose centers lie in $Z$. 
We call the elements of $\Div_{X,Z}$ the \defi{divisorial points} of $\Txz$.
We call the finite set $\Div_{X,Z}\big(\widetilde{\X}\big)$ of divisorial points of $\Txz$ the \defi{analytic boundary} of $\Txz$, and we denote it by $\boundaryZ$.
Since any formal log modification $\Y$ of $(X,Z)$ dominates $\widetilde{\X}$, we always have $\boundaryZ\subset\Div_{X,Z}(\Y)$.
}

\begin{ex}
If $X=\mathbb A^2_\C$ and $Z=\{0\}$, so that $\Txz$ is the valuative tree as in Example~\ref{valtree}, then $\boundaryZ$ consists of one point, corresponding to the exceptional divisor of the blow-up of the plane at the origin.
This is the $\I$-orbit of the order of vanishing at the origin of the complex plane, that is what Favre and Jonsson call the \defi{multiplicity valuation}.
\end{ex}

\begin{lem}\label{L: divisorial valuation specialization}
Let $\Y$ be a formal log modification of $(X,Z)$, and let $x$ be a closed point of $\Y\times_\X Z$. 
Then:
\begin{enumerate}
	\item $\Div_{X,Z}(\Y)$ is the inverse image via the specialization morphism $\Sp_\Y:\Txz\to \Y$ of the set of generic points of the irreducible components of $\Y\times_\X Z$;
	\item the open subset $\Sp_\Y^{-1}(x)$ of $\Txz$ can be given the structure of a pseudo-affinoid space.
\end{enumerate}
\end{lem}

\begin{proof}
If $\eta$ is the generic point of an irreducible component of $\Y\times_\X Z$, then the associated divisorial point specializes to $\eta$. 
Moreover, being a discrete valuation ring, $\OO_{\Y,\eta}$ is the only valuation ring dominating $\OO_{\Y,\eta}$, which means that there is only one point of $\Txz$ specializing to $\eta$, proving $(i)$.
To show $(ii)$, set $\mathscr U=\Spf\big(\widehat{\OO_{\Y,x}}\big)$. 
By~\ref{lemma fiber reduction}, the inverse image of $x$ in $\Y^{\beth}$ via the specialization morphism is isomorphic to $\mathscr U^{\beth}$. 
Let $f_x$ be a local equation for $\Y\times_\X Z$ at $x$; we then have $\Sp_\Y^{-1}(x)\cong \big(\mathscr U^\beth\setminus V(f_x)\big)/\I \cong {\mathscr U}_{f_x}^\beth$, and the latter is pseudo-affinoid.
\end{proof}

\pa{
We define a family $\calW$ of subsets of $\Txz$ as follows. 
Denote by $D=\widetilde{\Bl_{Z}X}\times_{X}Z$ the exceptional divisor in $\widetilde{\Bl_{Z}X}$.
A nonempty subset of $\Txz$ belongs to $\calW$ if and only if it is of the form $\Sp_{\widetilde{\X}}^{-1}(D\cap U)$, for some affine open $U$ of $\widetilde{\Bl_{Z}X}$ such that $D\cap U$ is a principal divisor of $U$.
}

\begin{lem}
The family $\calW$ is an atlas of $\Txz$.
\end{lem}

\begin{proof}
Let $W=\Sp_{\widetilde{\X}}^{-1}(D\cap U)$ be an element of $\calW$ corresponding to some affine open $U$ of $\widetilde{\Bl_{Z}X}$. 
Then $W=T_\calU$, where $\calU$ is the formal completion of $U$ along $D\cap U$. 
The formal scheme $\calU$ is affine since $U$ is affine, and it can be seen as a formal scheme $\calU_t$ of finite type over $k[[t]]$, where the $k[[t]]$-structure is defined by sending $t$ to an equation for $D\cap U$.
Therefore $T_\calU\cong\forg\big(\calU_t^\beth\big)$ is affinoid.
Moreover, since $\Sp_{\widetilde{\X}}^{-1}(D)=\Txz$ and $D$ is Cartier in $\widetilde{\X}$, the elements of $\calW$ cover $\Txz$ and therefore $\calW$ is an atlas of $\Txz$.
\end{proof}

\pa{\label{facts_canonical_atlas}
We call $\calW$ the \defi{canonical atlas} of $\Txz$. 
The reason this is relevant is the following. 
Let $V\subset \Txz$ be a $k((t))$-analytic space that is a union of elements of $\calW$.
Then $V$ can be written as a finite union of elements of $\calW$ since $D$ is quasi-compact, and the family of analytic subspaces $\calW|_V=\{W\in\calW \,\,|\,\,W\subset V\}$ is a distinguished formal atlas of $V$ (in the sense of \cite[\S1]{BosLut85}) by strict affinoid domains.
This means in particular that a formal model of $V$ can be reconstructed by gluing the affine formal $k[[t]]$-schemes of finite type $\Spf\big(\OOO_{W}(W)\big)$, for $W$ in $\calW|_V$.
Moreover, in our situation we can forget the $k[[t]]$-structures and glue all the affine special formal $k$-schemes $\Spf\big(\OOO_{W}(W)\big)$, retrieving the special formal $k$-scheme $\widetilde{\X}$. 
A similar idea will be used to construct formal log modifications in Theorem~\ref{vertex_sets}.
Note that the union of the Shilov boundary points of elements of $\calW$ is $\boundaryZ$ (see~\cite[\S2.4]{Ber90} for the definition of the Shilov boundary of an affinoid space).
In particular, if $V\subset \Txz$ is a $k((t))$-analytic space that is a union of elements of $\calW$, then $V\cap\boundaryZ$ is the analytic boundary of $V$ in the sense of~\cite[\S3.1]{Ber90}.
In the following we will often implicitly chose a structure of $k((t))$-analytic curve for the elements of $\calW$.
}

\begin{lem}\label{fibers_are_connected_components}
If $\Y$ is a formal log modification of $(X,Z)$, then the set of the connected components of $\Txz\setminus \Div_{X,Z}(\Y)$ coincides with the family
$$
\{\Sp_\Y^{-1}(x)\;|\;x\in \Y\times_\X Z\text{ closed point}\}.
$$
\end{lem}

\begin{proof}
Lemma~\ref{L: divisorial valuation specialization} implies that $\Txz\setminus \Div_{X,Z}(\Y)=\bigcup_x \Sp_\Y^{-1}(x)$, where this union, taken over the closed points of $\Y\times_\X Z$, is disjoint.
Each $\Sp_\Y^{-1}(x)$ is open by anticointinuity of $\Sp_\Y$, so it is a union of connected components of $\Txz\setminus \Div_{X,Z}(\Y)$. The fact that $\Sp_\Y^{-1}(x)$ is connected is \cite[6.1]{Bosch77} applied to any $k((t))$-analytic curve $W\in\calW$ such that $x\in\Sp_{\widetilde{\X}}(W)$.
\end{proof}

\pa{
If $C$ is a $k((t))$-analytic curve, its points can be divided into four types, according to the valuative invariants of their completed residue field (see e.g. \cite[3.3.2]{Duc}, although these ideas essentially go back to \cite{Ber90}). 
In particular a point $x$ of $C$ is said to be of \emph{type 2} if $\trdeg_{k} \widetilde{\rescompl{x}}=1$, where $\widetilde{\rescompl{x}}$ denotes the residue field of ${\rescompl{x}}$.
The points of type 2 are precisely the points of infinite branching of $C$, i.e. a point $x$ of $C$ is of type 2 if and only if $C\setminus\{x\}$ has infinitely many connected components (if $C$ is regular, this is equivalent to $C\setminus\{x\}$ having at least three connected components).
We refer to \cite[\S6]{Temkin15} or \cite{BakerPayneRabinoff14} for a description of the structure of non-archimedean analytic curves.
}

\begin{lem}\label{lemma_divisorial_iff_type2}
If $x$ is a point of $\Txz$, $V$ is an analytic domain of $\Txz$ that contains $x$, and $C$ is a $k((t))$-analytic curve such that $\forg(C)\cong V$, then $x$ is a divisorial point of $\Txz$ if and only if it is a point of type 2 of $C$.
\end{lem}

\begin{proof}
By abuse of notation we denote by $x$ also a point of $\Xeta$ whose image in $\Txz$ is the given point $x$.
The completed residue field $\rescompl{x}$ of $\Xeta$ at $x$ can be computed also as the completed residue field of $C$ at $x$.
Therefore, we deduce that it is a valued extension of $k((t))$ (for some non-trivial $t$-adic absolute value that we do not need to specify), and in particular 
\[
\rank_\Q|\rescompl{x}^\times|/|k^\times|\otimes_Z\Q\geq\rank_\Q|\rescompl{x}^\times|/|k((t))^\times|\otimes_Z\Q+1\geq1.
\]
Moreover, by Abhyankar's inequality (see \cite[Corollaire to 5.5]{Vaquie00}) we have 
\[
\rank_\Q|\rescompl{x}^\times|/|k^\times|\otimes_Z\Q+\trdeg_{k}\widetilde{\rescompl{x}}\leq 2.
\]
We said that $x$ is a type 2 point of $C$ if and only if $\trdeg_{k} \widetilde{\rescompl{x}}=1$, and by the two inequalities above this is equivalent to
\[
\begin{cases}
\rank_\Q|\rescompl{x}^\times|/|k^\times|\otimes_Z\Q=1 \\
\trdeg_{k}\widetilde{\rescompl{x}}=1
\end{cases}
\]
By \cite[Example 7, Proposition 10.1]{Vaquie00}, this is equivalent to $x$ being a divisorial point of $\Txz$.
\end{proof}


\pa{A \defi{vertex set} of $\Txz$ is any finite subset of $\Div_{X,Z}$ containing $\boundaryZ$.}

The following theorem is the main result of this section.

\begin{thm}\label{vertex_sets} 
Let $(X,Z)$ be as in \ref{setting_surfaces}.
Then the map $\Y\mapsto \Div_{X,Z}(\Y)$ induces an isomorphism between the following partially ordered sets:
\begin{enumerate}
 \item the set of isomorphism classes of formal log modifications of $(X,Z)$, ordered by domination;
 \item the set of vertex sets of $\Txz$, ordered by inclusion.
\end{enumerate}
\end{thm}

\begin{proof}
We follow the lines of \cite[6.3.15]{Duc}, but the general ideas (over an algebraically closed field) go back to \cite{BosLut85} and can be found also elsewhere, for example in \cite[\S4]{BakerPayneRabinoff14}. 
The proof will be divided in several steps.
In the first step, which is the hardest one, we rely on the proof of \cite[6.3.15]{Duc}, which is the corresponding statement for $k((t))$-analytic curves, by applying it to the elements of the canonical atlas $\calW$ of $\Txz$.
Given a vertex set of $\Div_{X,Z}$, we will first construct a special formal $k$-scheme by producing a suitable atlas of $\Txz$ and gluing the formal models of its elements.
The reader may check, using the definition of $\calW$, the facts discussed in \ref{facts_canonical_atlas}, and Lemma~\ref{fibers_are_connected_components}, that if we take $S=\boundaryZ$ then the atlas that we obtain will coincide with $\calW$.
This might be a useful example to keep in mind, noting that applying what follows in this case will yield the formal log modification $\widetilde\X$.
In the second step we will show that the formal scheme we obtained is a formal modification of $(X,Z)$.
We will then conclude the proof by showing that this association defines a bijection, and finally that it respects the given orderings.

\begin{step} {\it Construction of the formal scheme $\Y$.}
Let $S$ be a vertex set of $\Div_{X,Z}$, and let $\calV$ be the family of subsets of $\Txz$ defined as follows. 
A compact subset $V$ of $\Txz$ belongs to $\calV$ if and only if there exist a subset $S'$ of $S$ and a finite family $\{U_i\}$ of connected components of $\Txz\setminus S'$ such that the following conditions are satisfied:
\begin{enumerate}
 \item $V\subset W$ for some element $W$ of the canonical atlas $\calW$ of $\Txz$;
 \item $V=\Txz\setminus\coprod U_i$;
 \item $V\cap S=S'$;
 \item for every $x\in S'\setminus \boundaryZ$, there exists at least one index $i$ such that $x$ belongs to the topological boundary $\partial U_i\coloneqq \overline{U_i}\setminus U_i$ of $U_i$;
 \item every connected component of $\Txz$ that does not meet $S'$ is one of the $U_i$.
\end{enumerate}

\noindent
Observe that the empty set is an element of $\calV$, since it can be obtained by taking for $S'$ the empty set and for the family $\{U_i\}$ the set of connected components of $\Txz$.
To be able to use the elements of $\calV$ as building blocks for the formal scheme $\Y$, we will now prove that $\calV$ is closed under intersection.
Indeed, if $V_1$ and $V_2$ are elements of $V$ corresponding respectively to families $\{U_{1,i}\}$ and $\{U_{2,j}\}$ of subspaces of $\Txz$, then $V_3=V_1\cap V_2$ is the element of $\calV$ corresponding to the subset $S_3=S\cap V_3$ of $S$ and the family of those connected components of $\Txz\setminus S_3$ that can be written as unions of sets of the form $U_{1,i}$ or $U_{2,j}$.
The only part which is non-trivial to verify is the fact that these data satisfy the condition (v) above.
For this, assume that $U$ is a connected component of $\Txz$ which contains no point of $S_3$.
Then $U$ can not be entirely contained in $V_3$, or otherwise it would be contained in $V_1$ as well, contradicting condition (v) for $V_1$.
Therefore if $U$ intersects $V_3$ non-trivially there exists a point $x$ contained both in $U$ and in the topological boundary $\partial V_3=V_3\setminus\mathrm{Int}(V_3)$ of $V_3$.
But then $x$ would also belong to the topological boundary of one of the first two $V_i$, hence to $S_i$, and so as it belongs to $V_3$ it would be an element of $S_3$ as well.
Since this is not possible, $U$ does not intersect $V_3$, which proves (v) because $U$ is also a connected component of $\Txz\setminus S_3$.
Now let $W$ be a connected element of $\calW$, seen as a $k((t))$-analytic space, and consider the family $\calV|_W=\{V\cap W \,\,|\,\,V\in\calV\}$ of subspaces of $W$, seen as $k((t))$-analytic subspaces themselves.
Oberve that, if $V$ is an element of $\calV$ associated as above with a family $\{U_i\}$ of subsets of $\Txz$, then the element $V'=V\cap W$ of $\calV|_W$ satisfies itself the conditions (i--v) above with respect to the ambient space $W$ and the family of subsets of $W$ that are connected components of subsets of the form $U_i\cap W$, where in condition (v) $\boundaryZ$ is replaced by $\boundaryZ\cap W$.
Then, since as we observed in \ref{facts_canonical_atlas} the Shilov boundary of $W$ coincides with $\boundaryZ\cap W$, we can apply \cite[6.3.15.2]{Duc} (whose hypotheses (a--d) now follow directly from our (i--v) above) to $W$ and to the family $\calV|_W$, deducing that $\calV_W$ is a strict formal affinoid atlas of $W$, and moreover the associated vertex set (that is by definition the union of the Shilov boundaries of the elements of $\calV_W$) is $S\cap W$.
The associated formal $k[[t]]$-scheme $\Y_W$ is therefore a formal model of $W$ with vertex set $S\cap W$.
Now, observe that the canonical model of an affinoid domain $V$ of $\Txz$, being $\Spf\big(\OOO_{\Txz}(V)\big)$, does not depend on the choice of a $k((t))$-analytic structure on $V$.
This guarantees that we can glue all the $\Y_W$, seen as affine special formal $k$-schemes, along their intersections, obtaining a special formal $k$-scheme $\Y$.
\end{step}

\begin{step} {\it $\Y$ is a formal log modification of $(X,Z)$.}
We defined the formal scheme $\Y$ by gluing affine special formal $k$-schemes of the form $\Spf\big(\OOO_{\Txz}(V)\big)$, for $V$ ranging among the elements of $\calV$.
By \cite[2.1]{MartinKappen15} (as recalled in~\ref{dejong_original}) each $\OOO_{\Txz}(V)$ is integrally closed in its ring of fractions, and therefore $\Y$ is normal.
For each $W\in\calW$, the inclusions of $k((t))$-analytic spaces $V\to W$, for $V\in\calV|_W$, induce morphisms of special $k$-algebras $\OOO_{\Txz}(W)\to\OOO_{\Txz}(V)$, and therefore a morphism of special formal $k$-schemes $\Y_W\to \widetilde \X$.
These morphisms glue to an adic morphism $\Y\to\widetilde \X$, so we obtain an adic morphism $f\colon\Y\to\X$, and $\Y\times_\X Z$ is Cartier in $\Y$ by the universal property of the blowup. 
Moreover, since for every $W$ in the covering $\calW$ the morphism $\Y_W\to \widetilde \X$ induces an isomorphism at the level of normalized spaces $T_{\Y_W}\cong W$, the morphism $f$ induces an isomorphism $T_\Y\cong\Txz$.
Therefore, $\Y$ is a formal log modification of $(X,Z)$.
\end{step}

\begin{step} {\it Bijectivity of the correspondence.}
It follows from our construction in the first step that each element $V$ of $\calV$ is an affinoid domain of $\Txz$ and coincides with the inverse image under the specialization morphism $\Sp_\Y$ of an affine open subset of $\Y$ in which $\Y\times_\X Z$ is principal.
It follows from \cite[2.4.4]{Ber90}, applied after choosing a $k((t))$-analytic structure on $V$, that $V\cap S$, being the Shilov boundary of $V$, coincides with $V\cap\Div_{X,Z}(\Y)$.
Therefore $S$, being the union of the Shilov boundaries of the elements of $\calV$, coincides with $\Div_{X,Z}(\Y)$.
This shows that the map $\Y\mapsto \Div_{X,Z}(\Y)$ is surjective.
To prove its injectivity, we need to show that a formal log modification of $(X,Z)$ is determined by its divisorial set. 
This can be done locally, again using Ducros's results, as follows. 
Assume that $\Y'$ is another formal log modification of $(X,Z)$ such that $\Div_{X,Z}(\Y')=S$.
Let $W$ be the element of the canonical atlas $\calW$ associated with an open affine subspace $U$ of $\widetilde{\X}$. 
Then by~\cite[6.3.15]{Duc} $\calV|_W$ is the unique formal atlas on $W$ whose vertex set is $S\cap W$, therefore the $k[[t]]$-subspace $\tau^{-1}(U)$ of $\Y'$, where we denote by $\tau$ the composition of the canonical map $\Y'\to\widetilde{\X}$ with $\Sp_{\Y'}$, is isomorphic to the open $\Y_W$ of $\Y$, and hence $\Y'$ is isomorphic to $\Y$.
%
%
\end{step}

\begin{step} {\it Functoriality.}
It is clear that if $\Y$ and $\Y'$ are two formal log modifications of $(X,Z)$ such that $\Y'$ dominates $\Y$, then $\Div_{X,Z}(\Y)\subset \Div_{X,Z}(\Y')$.
To show that the bijective correspondence that we have constructed respects the partial orders it is then enough to note the following.
Let $S_1\subset S_2$ be finite nonempty subsets of $\Div_{X,Z}$, and let $\Y_1$ and $\Y_2$ be the corresponding formal models, defined using formal atlases $\mathscr V_1$ and $\mathscr V_2$. 
Then from the definition of the atlases $\mathscr V_i$ it follows that we can cover $\Txz$ by $V_{1,1},\ldots,V_{1,r}\in\mathscr V_1$ and also by $V_{2,1},\ldots,V_{2,s}\in\mathscr V_2$ in such a way that each $V_{2,i}$ is a subspace of some $V_{1,i}$, and each $V_{1,i}$ is covered by the $V_{2,i}$'s that it contains. 
These inclusions give a morphism $\Y_2\to\Y_1$ commuting with the two morphisms $\Y_i\to\X$, hence a morphism of formal log modifications.
\end{step}

 This completes the proof of Theorem~\ref{vertex_sets}.
\end{proof}

\Pa{Remarks}{
If $X$ has only rational singularities or if $k$ is an algebraic closure of the field $\mathbb F_p$ for some prime number $p$, then Theorem~\ref{vertex_sets} can also be proved using resolution of singularities to find a suitable log modification and then contractibility results \cite[2.3, 2.9]{Artin62} to contract all unnecessary divisors, and every formal log modification of $(X,Z)$ is algebraizable. 
In general not all of the formal log modifications given by Theorem~\ref{vertex_sets} are algebraizable, but the contractibility criterion of Grauert-Artin \cite{Artin70} guarantees that they can always be algebraized in the category of algebraic spaces over $k$.
Moreover, since Artin proved that a smooth algebraic space in dimension 2 is a scheme, we retrieve Proposition~\ref{proposition_algebraization_resolutions}.
}


\section{Discs and annuli}
\label{section_discs}\label{section_4.2}

In this section we will study one-dimensional open discs and open annuli in normalized spaces. 
The main result, Proposition~\ref{prop_discs}, explains in which sense those discs and annuli are determined by their canonical reduction.

\pa{
We say that a $k((t))$-analytic space $X$ is \defi{pseudo-affinoid} if it is the Berkovich space associated with an affine special formal $k[[t]]$-scheme. 
When this is the case and moreover $X$ is reduced, \cite[7.4.2]{deJ95} (recalled in~\ref{dejong_original}) tells us that $X\cong\X_t^\beth$, where $\X_t=\Spf\big({\OOO_X(X)}\big)$ is called the \emph{canonical formal model} of $X$, and moreover $\X_t$ is integrally closed in its generic fiber. 
The reduced affine special formal $k$-scheme $\big((\X_t)_s\big)_{\mathrm{red}}$ associated with the special fiber $(\X_t)_s=\X_t\otimes_{k[[t]]} k$ of $\X_t$ will then be called \emph{canonical reduction} of $X$, and will be denoted by $X_0$.
\\
We say that a pseudo-affinoid $k((t))$-analytic space $X$ is \defi{distinguished} if the special fiber $(\X_t)_s$ of its canonical formal model is already reduced, i.e. if it coincides with $X_0$.
We have shown in \ref{distinguished_affinoid_equivalence} that for strictly affinoid $k((t))$-analytic spaces this definition coincides with the classical one.
}

\pa{
We say that a normalized $k$-space $Y$ is \emph{pseudo-affinoid} (respectively, \emph{distinguished pseudo-affinoid}) if it is of the form $\forg(X)$, with $X$ a pseudo-affinoid (resp. a distinguished pseudo-affinoid) $k((t))$-analytic space.
This coincides with the definition in~\ref{definition_pseudo-affinoid}. 
Whenever $Y$ is reduced, the affine special formal $k$-scheme $\Y=\Spf\big({\OOO_Y(Y)}\big)$ is called the \emph{canonical formal model} of $Y$. 
Note that this is an abuse of notation since $\Y$ is not a formal model of the normalized $k$-space $Y$.
\\
We define the \emph{canonical reduction} of $Y$ as the closed formal subscheme $Y_0$ of $\Y$ defined by the ideal $I=\bigcap\sqrt{(f)}$, where the intersection is taken over all elements $f$ of $\OOOO_Y(Y)$ that do not vanish on $Y$. 
Being an intersection of radical ideals, $I$ is radical itself, so $Y_0$ is a reduced special formal $k$- scheme.
This definition is consistent with the previous one, since $Y_0$ is isomorphic to the canonical reduction $X_0$ of $X$, and therefore the canonical reduction of a pseudo affinoid $k((t))$-analytic space only depends on its normalized space structure.
To prove this, we need to show that $I=\sqrt{(t)}$. 
Clearly $I\subset\sqrt{(t)}$ since $t$ does not vanish on $Y$. 
By Theorem~\ref{thm_structure_arboretum} we have $Y=\forg(X)=T_\X\setminus V(t)$, where $\X_t$ is the canonical $k[[t]]$-formal model of $X$, as usual $\X$ is the underlying special formal $k$-scheme, and by abuse of notation we denoted by $t$ the image of $t$ in $\OO_\X(\X)$. 
Therefore, if $f$ does not vanish on $Y$ we must have $V(f)\subset V(t)$, and so $\sqrt{(t)}\subset\sqrt{(f)}$, which implies that $\sqrt{(t)}\subset I$.
Moreover, remark that $Y$ is distinguished if and only if the ideal $I$ defined above is a principal ideal.
}

\pa{
A $k((t))$-analytic space is called an \defi{open $k((t))$-disc}, or simply a \emph{disc}, if it is isomorphic to $\Spf\big(k[[t]][[T]]\big)^{\beth}_t$.
\\
Equivalently, a disc is a $k((t))$-analytic space isomorphic to the subspace of $\mathbb A^{1,\mathrm{an}}_{k((t))}$ defined by the inequality $|T|<1$, where $\mathbb A^1_{k((t))}=\Spec(k((t))[T])$.
}

\pa{
A $k((t))$-analytic space is called an \defi{open $k((t))$-annulus of modulus $n$}, or simply an \emph{annulus of modulus $n$}, if it is of the form $A_n\coloneqq \Spf\big(k[[t]][[T_1,T_2]]/(T_1T_2-t^n)\big)_t^{\beth}$, for some $n>0$. 
\\
Equivalently, an annulus of modulus $n$ is a $k((t))$-analytic space isomorphic to the subspace of $\mathbb A^{1,\mathrm{an}}_{k((t))}$ defined by the inequality $|t^n|<|T_1|<1$. 
\\
We define a \defi{standard annulus} as an annulus of modulus one. 
Remark that an annulus is standard if and only if it has no $k((t))$-point.
}

\pa{\label{modulus_intrinsic}
The modulus of an annulus $X$ is well defined, and depends only on the algebra $\OOO_X(X)$, hence only on the normalized space $\forg(X)$.
Indeed, if $X$ has modulus $n$ then $\OOO_X(X)$ is the completed local ring of a $k$-surface at a du Val singularity of type $A_{n-1}$, and there is exacly one formal isomorphism class of surfaces with singularity of type $A_{n-1}$ (see for example \cite[\S2]{Artin77}).
}

\pa{
Discs and annuli are distinguished pseudo-affinoid $k((t))$-analytic spaces. 
Indeed, both are reduced and the canonical formal model of a disc is the affine formal $k[[t]]$-scheme $\Spf\big(k[[t]][[T]]\big)_t$, whose special fiber is $\Spf\big(k[[T]]\big)$; while the canonical formal model of an annulus of modulus $n$ is the affine formal $k[[t]]$-scheme $\Spf\big(k[[t]][[T_1,T_2]]/(T_1T_2-t^n)\big)_t$, whose special fiber is $\Spf\big(k[[T_1,T_2]]/(T_1T_2)\big)$.
}

\begin{rem}
The canonical model of an annulus is regular if and only if the annulus is standard. 
Indeed, the maximal ideal of $k[[t,T_1,T_2]]/(T_1T_2-t^n)$ is $\frakM=(t,T_1,T_2)$, hence $\frakM^2=(t^2,T_1^2,T_2^2,tT_1,tT_2,t^n)$ so the $k$-vector space $\frakM/\frakM^2$ has dimension 2, with basis $\{T_1,T_2\}$, if and only if $n=1$.
\end{rem}

It is clear that any two $k((t))$-discs are always isomorphic as $k((t))$-analytic spaces, and that two $k((t))$-annuli are isomorphic if and only if they have the same modulus.
In our setting we need something stronger, which will be the content of the next proposition and of the corollary following it.

\begin{prop}\label{prop_discs}
Let $X$ be a distinguished pseudo-affinoid $k((t))$-analytic space, and denote by $X_0$ its canonical reduction. 
Then:
\begin{enumerate}
\item if $X_0\cong\Spf\big(k[[T]]\big)$, then $X$ is a $k((t))$-disc;
\item if $X_0\cong\Spf\big(k[[T_1,T_2]]/(T_1T_2)\big)$ and $X$ is irreducible, then $X$ is a $k((t))$-annulus.
\end{enumerate}
\end{prop}

\begin{proof}
Part $(i)$ follows easily from the uniqueness of deformations of smooth affine formal schemes, see \cite{PerezRodriguez08}. 
Indeed, up to isomorphism there is only one affine and flat special formal $k[[t]]$-scheme whose special fiber is $\Spf\big(k[[T]]\big)$, so the canonical formal model of $X$ is isomorphic to $\Spf\big(k[[t]][[T]]\big)$, hence $X$ is a $k((t))$-disc.
To prove $(ii)$ we make use of the fact that the {miniversal} deformation of the formal node $\Spf\big(k[[T_1,T_2]]/(T_1T_2)\big)$ is $\Spf\big(k[[t,T_1,T_2]]/(T_1T_2-t)\big)$, which is \cite[14.0.1]{Hartshorne10}.
This means that if $\X=\Spf(A)$ is an affine and flat special formal $k[[t']]$-scheme whose special fiber is isomorphic to $\Spf\big(k[[T_1,T_2]]/(T_1T_2)\big)$, then there is a local $k$-algebra morphism $\varphi\colon k[[t]]\to k[[t']]$ such that $\X \cong \Spf\big(k[[t,T_1,T_2]]/(T_1T_2-t)\big)\otimes_{k[[t]]}k[[t']]$.
The morphism $\varphi$ is determined by a power series $\varphi(t)=F(t')\in k[[t']]$ such that $F(0)=0$, so we have $\X\cong\Spf\big(k[[t',T_1,T_2]]/(T_1T_2-F(t'))\big)$.
Moreover, since $X$ is irreducible then $F(t')$ cannot be zero.
The power series $F(t')$ can be written as $u(t')^n$ for some unit $u$ of $k[[t']]$, hence by further sending $T_1$ to $uT_1$ we obtain an isomorphism $\X\cong\Spf\big(k[[t',T_1,T_2]]/(T_1T_2-(t')^n)\big)$. 
By applying this to the canonical $k[[t]]$-formal model of $X$, we deduce that $X$ is a $k((t))$-annulus.
\end{proof}

\begin{cor}\label{cor_discs}
Let $X$ be a $k((t))$-disc or a $k((t))$-annulus and let $Y$ be a distinguished pseudo-affinoid $k((t))$-analytic space such that  $\forg(X)\cong \forg(Y)$. 
Then $X$ and $Y$ are isomorphic as $k((t))$-analytic spaces.
\end{cor}

\begin{proof}
This follows from Proposition~\ref{prop_discs} since the special fiber of the canonical formal model of a distinguished pseudo-affinoid $k((t))$-analytic space does not depend on the chosen distinguished pseudo-affinoid $k((t))$-analytic structure. 
For annuli, note that the modulus of $Y$ is the same as the modulus of $X$ by~\ref{modulus_intrinsic}.
\end{proof}

\pa{
The last result allows us to define more intrinsically discs and annuli in normalized spaces: we say that a distinguished pseudo-affinoid analytic domain $V$ of a normalized $k$-space $T$ is a \defi{disc} (respectively an \defi{annulus of modulus $n$}) if there exists a $k((t))$-analytic space $X$ such that $V\cong\forg(X)$ and $X$ is a disc (resp. an annulus of modulus $n$). 
Corollary~\ref{cor_discs} tells us that this property is independent of the choice of a distinguished pseudo-affinoid structure on $V$.
}


\section{Formal fibers}
\label{section_formal_fibers}\label{section_4.3}

In this section we move to the study of the fibers of the specialization morphism. 
For normal surfaces we will get in Proposition~\ref{formal_fibers} very explicit results involving discs and annuli, analogous to \cite[2.2 and 2.3]{BosLut85}. 
For simplicity, from now on we assume that $k$ is algebraically closed.

\pa{Let $\X$ be a special formal $k$-scheme and let $x$ be a point of $\X$. 
We define the \defi{formal fiber} of $x$ as the inverse image $\calF_x=\Sp_\X^{-1}(x)$ of $x$ in $\T$ under the specialization morphism. 
It is a subspace of $\T$, open if $x$ is closed in $\X$.
}

\pa{ 
Let $X$ be a normal scheme of finite type over $k$, let $Z$ be a divisor of $X$, and let $\X=\widehat{X/Z}$ the formal completion of $X$ along $Z$.
Then the argument of Lemma~\ref{L: divisorial valuation specialization} tells us that, if $\eta$ is a generic point of an irreducible component of $Z$, its formal fiber $\calF_\eta$ is a single point of $\T$, precisely the point corresponding to the $\I$-orbit of divisorial valuations associated with the component $\overline{\{\eta\}}$.
If $x$ is a closed point of $\X$, its formal fiber $\calF_x$ can be given the structure of a pseudo-affinoid space.
}

\begin{lem}\label{L: regular formal fiber}
Let $\X$ be a normal special formal $k$-scheme of dimension $n$ and let $x$ be a closed point of $\X$ such that there exists an ideal of definition of $\X$ that is principal at $x$.
Then $\X$ is regular at $x$ if and only if $\OOO_\T(\calF_x)\cong k[[X_1,\ldots,X_n]]$.
\end{lem}

\begin{proof}
Consider the normal special formal $k$-scheme $\mathscr U=\Spf\big(\widehat{\OO_{\X,x}}\big)$.
By \ref{lemma fiber reduction}, the inverse image of $x$ in $\X^{\beth}$ via the specialization morphism is isomorphic to $\mathscr U^{\beth}$, so $\calF_x$ is isomorphic to $(\mathscr U^{\beth}\setminus V(t))/\I\cong\mathscr U_t^\beth$, where we have denoted by $t$ a local generator of an ideal of definition of $\X$ at $x$.
It follows that  $\OOO_\T(\calF_x)\cong \OOO_{\mathscr U_t^\beth}(\mathscr U_t^\beth)$, and since $\mathscr U_t$ is normal this is also equal to $\OO_{\mathscr U_t}(\mathscr U_t)\cong \widehat{\OO_{\X,x}}$ by \cite[7.3.6]{deJ95}.
To conclude, by Cohen's structure theorem~\cite{Cohen46} $x$ is regular in $\X$ if and only if $\widehat{\OO_{\X,x}}\cong k[[X_1,\ldots,X_n]]$.
\end{proof}

\begin{rem}
Note that in the lemma above, while the $k$-algebra $\OOO_\T(\calF_x)$ does not depend on the geometry of $\X_0$ around $x$, its largest ideal of definition, and therefore also the space $\calF_x$, strongly depends on it.
Focusing now on the case of surfaces, an example of this behavior is detailed in the following proposition, which is the analogue for normalized spaces of surfaces of a classical result of Bosch and L\"utkebohmert \cite[2.2 and 2.3]{BosLut85} (see also \cite[4.3.1]{Ber90} for a formulation in the language of Berkovich curves).
\end{rem}

\begin{prop}\label{formal_fibers}
Let $\X$ be a normal special formal $k$-scheme of dimension $2$ and let $x$ be a closed point of $\X$ such that there exists an ideal of definition of $\X$ that is principal at $x$.
Then:
\begin{enumerate}
 \item $x$ is regular both in $\X$ and in $\X_0$ if and only if its formal fiber $\calF_x$ is a disc;
 \item $x$ is regular in $\X$ and an ordinary double point in $\X_0$ if and only if $\calF_x$ is a standard annulus.
\end{enumerate}
\end{prop}

\begin{proof}
We endow $\calF_x$ with the pseudo-affinoid $k((t))$-analytic structure from Lemma~\ref{L: divisorial valuation specialization}.
Assume that $x$ is regular both in $\X$ and in $\X_0$. 
Then $\X_0$ is itself principal locally at $x$, 
 and so we can choose an element $t$ of $\OO_{\X,x}$ that locally defines $\X_0$.
Then the image $\bar t$ of $t$ in $\frakM/\frakM^2$ does not vanish, where we have denoted by $\frakM$ the maximal ideal of $\OO_{\X,x}$.
We then pick another element $X$ of $\frakM$ whose image in $\frakM/\frakM^2$ generates it together with $\bar t$.
Cohen Theorem gives us an isomorphism $\widehat{\OO_{\X,x}}\cong k[[t,X]]$, therefore the formal fiber $\calF_x$ is the disc $\mathcal{F}_x=\Spf\big(k[[t,X]]\big)_{t}^\beth$.
Conversely, if $\calF_x$ is a disc, isomorphic to $\Spf\big(k[[t,X]]\big)^{\beth}_t$, then $\OOO_{\Txz}(\calF_x) \cong k[[t,X]]$, so $\X$ is regular at $x$ by Lemma~\ref{L: regular formal fiber}, and $\X_0$ is locally defined by the equation $t=0$, so it is itself regular, proving $(i)$.
The proof of $(ii)$ is similar.
If $x$ is an ordinary double point of $\X_0$ then we can find elements $X$ and $Y$ of $\frakM$ whose images in $\frakM/\frakM^2$ generate it and such that $\X_0$ is defined at $x$ by the equation $XY=0$.
Then we obtain $\calF_x=\Spf\big(k[[X,Y]]\big)_{XY}^\beth=\Spf\big(k[[t]][[X,Y]]/(XY-t)\big)_t^\beth$, which is a standard annulus.
The converse implication is proved as in $(i)$, as if $\calF_x$ is a standard annulus then $\OOO_{\Txz}(\calF_x) \cong k[[t]][[X,Y]]/(XY-t) \cong k[[X,Y]]$.
\end{proof}


\section{Log essential and essential valuations}
\label{section_4.4}

In this section we make use of the previous results to characterize in terms of the structure of $\Txz$ the finite sets of divisorial points of $\Txz$ that correspond via Theorem~\ref{vertex_sets} to the log resolutions of a pair $(X,Z)$.
We will then be able to describe the divisorial points corresponding to the minimal log resolution of $(X,Z)$, and in Theorem~\ref{thm_log_essential_valuations} we will deduce a local characterization of the log essential valuations of $(X,Z)$.
Finally, we will show that similar arguments also yield a characterization of the slightly smaller class of essential valuations which has been introduced by Nash in \cite{Nash95} and studied extensively since.
This is the content of Theorem~\ref{thm_essential_valuations}.

Note that in this section we use the existence of resolution of singularities for surfaces. 
More precisely, we admit the fact that given any set $S$ of divisorial points of $\Txz$ there exists a log resolution $Y$ of $(X,Z)$ such that $S\subset \Div_{X,Z}(Y)$. 
On the other hand, we do not need to assume the existence of a minimal (log) resolution of $(X,Z)$.

\pa{Let $(X,Z)$ be as in \ref{setting_surfaces}.
We say that a vertex set $S$ of $\Txz$ is \defi{regular} if the connected components of $\Txz\setminus S$ are discs and a finite number of standard annuli.
}

Putting together Lemma~\ref{fibers_are_connected_components}, Proposition~\ref{formal_fibers} and Proposition~\ref{proposition_algebraization_resolutions} we obtain the following important result.

\begin{prop}\label{prop_resolution_iff_vertex_set}
Let $k$ be an algebraically closed field and let $(X,Z)$ be as in \ref{setting_surfaces}.
Then a formal log modification $\Y$ of $(X,Z)$ is a log resolution if and only if $\Div_{X,Z}(\Y)$ is a regular vertex set of $\Txz$.
\end{prop}
%

\pa{\label{observation_fibers_are_simple}
We define the set of \defi{log essential valuations} of the pair $(X,Z)$ as the intersection $\bigcap_Y\Div_{X,Z}(Y)$, where $Y$ ranges among the log resolutions of $(X,Z)$.
In words, it is the set of divisorial points of $\Txz$ whose center on every log resolution of $(X,Z)$ is a divisor; if we assume the existence of a minimal log resolution $Y_\mathrm{min}$ of $(X,Z)$ then it coincides with the set $\Div_{X,Z}(Y_\mathrm{min})$. 
We call an open subset $U$ of $\Txz$ \defi{simple} if it is isomorphic to either a disc or a standard annulus, $U\cap\boundaryZ=\emptyset$ and the topological boundary $\partial U=\overline U\setminus U$ is contained in $\Div_{X,Z}$. 
Then a finite subset $S$ of $\Txz$ is a regular vertex set if and only if all the connected components of $\Txz\setminus S$ are simple subspaces of $\Txz$.
}

\begin{thm}\label{thm_log_essential_valuations}
Let $k$ be an algebraically closed field, let $(X,Z)$ be as in \ref{setting_surfaces} and let $v$ be an element of $\Div_{X,Z}$. 
Then $v$ is a log essential valuation of $(X,Z)$ if and only if it has no simple neighborhood in $\Txz$.
\end{thm}

\begin{proof}
Proposition~\ref{prop_resolution_iff_vertex_set} implies that a divisorial point which has no simple neighborhood in $\Txz$ is log essential.
To prove the reverse implication, assume that $U$ is a simple subspace of $\Txz$ and that $v$ is a divisorial point contained in $U$.
Since $\partial U$ is a finite set of divisorial points, there exists a log resolution $Y$ of $(X,Z)$ such that $\Div_{X,Z}(Y)$ contains $\partial U$.
Set $S=\Div_{X,Z}(Y)\setminus U$.
Then $S$ is a finite subset of $\Div_{X,Z}$, nonempty because it contains $\partial U$, so by Theorem~\ref{vertex_sets} it is of the form $\Div_{X,Z}(\Y')$ for some formal log modification $\Y'$ of $(X,Z)$.
The connected components of $\Txz\setminus \Div_{X,Z}(\Y')$ are either $U$ or connected components of $\Txz\setminus \Div_{X,Z}(Y)$, so they are all simple.
By Proposition~\ref{prop_resolution_iff_vertex_set} the formal log modification $\Y'$ can therefore be algebraized to a log resolution $Y'$ of $(X,Z)$.
Since $v$ doesn't belong to $\Div_{X,Z}(Y')$, this contradicts the fact that $v$ is log essential.
\end{proof}

\begin{rem}\label{J4_claim_boundary}
Any element of the cover discussed in Example~\ref{analytic_structure_valuative_tree} provides an example in the valuative tree $T_{\mathbb A^2_\mathbb{C},0}$ of a disc whose complement is a nondivisorial point. 
This shows that it is really necessary to impose the condition on the topological boundary in the definition of simple subspace.
\end{rem}

We will now move to studying the more classical concept of essential valuations.

\pa{
Let us assume that $k=\C$ and that $Z=X_{sing}$ is the singular locus of $X$.
Then the set of log essential valuations contains the set of essential valuations studied by Nash in~\cite{Nash95}. 
More generally, over an arbitrary algebraically closed field the set of \defi{essential valuations} of a pair $(X,Z)$ is defined as the intersection $\bigcap_Y\Div_{X,Z}(Y)$, where $Y$ ranges among the (not necessarily log) resolutions of $(X,Z)$ and $\Div_{X,Z}(Y)$ denotes the finite set of divisorial points associated to those components of the exceptional locus of $Y\to X$ which are divisors.
The sets of essential and log essential valuations differ when the minimal resolution of the pair $(X,Z)$ is not a log resolution, i.e. when its exceptional locus is not a normal crossings divisor. 
An example is given in \ref{example_nonlogessential2}.
However, essential and log essential valuations coincide for big classes of singularities, for example for rational singularities.
}

\pa{
We call an open subset $U$ of $\Txz$ \emph{elementary} if its ring of bounded analytic functions $\calO^\circ_{\Txz}(U)$ is isomorphic to $k[[t,u]]$ and its topological boundary $\partial U=\overline{U} \setminus U$ is a finite subset of $\Div_{X,Z}$.
Observe that every simple subset of $\Txz$ is elementary; a crucial difference is that elementary subsets can contain points of $\boundaryZ$.
}

We can now give a local criterion for essential valuations analogous to Theorem~\ref{thm_log_essential_valuations}.
For technical reasons we will restrict ourselves to the case of a normal surface singularity $x\in X$.
This is the case that is usually considered in the literature on the Nash problem.

\begin{thm}\label{thm_essential_valuations}
Let $k$ be an algebraically closed field, let $X$ be a normal surface over $k$, let $x$ be a point of $X$, and let $v$ be an element of $\Div_{X,x}$. 
Then $v$ is an essential valuation of $(X,x)$ if and only if it has no elementary neighborhood in $T_{X,x}$.
\end{thm}

The proof of this result is analogous to the one of Theorem~\ref{thm_log_essential_valuations}.
Before we give it, we need a proposition which combines some results similar to Lemma~\ref{fibers_are_connected_components}, Theorem~\ref{vertex_sets}, Proposition~\ref{formal_fibers} and Proposition~\ref{prop_resolution_iff_vertex_set}.

\begin{prop}\label{prop_resolution_elementary_components}
Let $k$ be an algebraically closed field, let $X$ be a normal surface over $k$, let $x$ be a point of $X$, and let $S$ be a finite nonempty subset of $\Div_{X,x}$.
Then there exist a normal special formal $k$-scheme $\Y$ and an adic morphism $\Y \to \widehat{X/x}$ inducing an isomorphism of normalized spaces such that the associated set of divisorial points $\Div_{X,x}(\Y)$ is $S$.
If moreover every connected component of $\Div_{X,x}\setminus S$ is elementary, then $\Y$ can be algebraized by a resolution $Y$ of $(X,x)$.
\end{prop}

\begin{proof}
Choose a resolution $(Y,D)$ of $(X,x)$ such that $S\subset \Div_{X,x}(Y)$.
By the contractibility criterion of Grauert-Artin \cite{Artin70} we can contract every component of $D$ which does not correspond to an element of $S$, yielding a normal algebraic spaces $\calY$ over $k$ with a proper morphism $f$ to $X$ (for a systematic treatment of algebraic spaces we refer the reader to \cite{Knutson1971}).
Indeed, the intersection matrix of the divisor that we want to contract is negative definite because the entire exceptional divisor of $Y$ can be contracted to $x$ in $X$.
By taking the formal completion of this algebraic space along $f^{-1}(x)$ we obtain the formal $k$-scheme $\Y$ that we want.
The connected components of $T_{X,x}\setminus \Div_{X,x}(\Y)$ are the inverse images through the center map of the closed points of $\Y_0$ (this can be proven as in Lemma~\ref{fibers_are_connected_components}).
Let $y$ be a closed point of $\Y_0$ and let $\mathcal I_y$ be the image of the ideal defining $\Y_0$ in $\mathcal O_{\Y,y}$.
As in Lemma~\ref{L: divisorial valuation specialization}, it follows from \ref{lemma fiber reduction} that we have a canonical isomorphism $\Sp_\Y^{-1}(y)\cong T_{\Spf\big(\widehat{\mathcal O_{\Y,y}}\big)}\setminus V(\mathcal I_y)$.
Therefore we have 
\begin{align*}
\OOO_{T_{X,x}}\big({\Sp_\Y^{-1}(y)}\big)
& \cong
\OOO_{T_{\Spf(\widehat{\mathcal O_{\Y,y}})}}\Big(T_{\Spf(\widehat{\mathcal O_{\Y,y}})}\setminus V(\mathcal I_y)\Big) 
\\
& \cong
\OOO_{T_{\Spf(\widehat{\mathcal O_{\Y,y}})}}\Big(T_{\Spf(\widehat{\mathcal O_{\Y,y}})}\Big)
\cong
\widehat{\mathcal O_{\Y,y}},
\end{align*}
where the second isomorphism follows from the extension theorem~\cite[Proposition 3.3.14]{Ber90} and the third one holds because 	$\Y$ is normal at $y$.
If $\Sp_\Y^{-1}(y)$ is elementary, it follows then from Cohen's theorem that $\Y$ is smooth at $y$.
Since this holds for every $y$, $\Y$ is non-singular, so the algebraic space $\calY$ is a non-singular, separated two-dimensional algebraic space over a field, hence it can be algebraized by a scheme (see \cite[V.4.9,10]{Knutson1971}), yielding the resolution $Y$ that we are looking for.
\end{proof}

\Pa{Remarks}{
If we are working over the field of complex numbers, we can apply Grauert contractibility criterion \cite{Grauert62} instead of Artin's and obtain $\Y$ as a complex analytic space.
Of course this is the same as analytifying the complex algebraic space given by Artin's criterion.
Observe that the surface $Y$ above is not necessarily a log resolution of $(X,x)$, as the exceptional locus of the morphism $Y\to X$ may not be a divisor with normal crossings.
}

\begin{proof}[Proof of Theorem~\ref{thm_essential_valuations}]
As before, the reasoning of Proposition~\ref{prop_resolution_elementary_components} implies that a divisorial valuation which has no elementary neighborhood in $T_{X,x}$ is essential.
To prove the reverse implication, let $U$ be an elementary subspace of $T_{X,x}$ and let $v$ be a divisorial point contained in $U$.
Let $Y$ be a log resolution of $(X,x)$ such that $\Div_{X,x}(Y)$ contains $\partial U$, and set $S=\Div_{X,x}(Y)\setminus U$.
Then $S$ is finite and nonempty, and the connected components of $\Div_{X,x}\setminus S$, being either $U$ or connected components of $T_{X,x}\setminus \Div_{X,x}(Y)$, are all elementary.
Therefore, Proposition~\ref{prop_resolution_elementary_components} tells us that there exists a resolution $Y'$ of $(X,x)$ such that $\Div_{X,x}(Y')=S$.
This proves that $v$ is not an essential valuation, since it doesn't belong to $\Div_{X,x}(Y')$.
\end{proof}

%
\begin{ex}\label{example_nonlogessential2}
Let us give an example of a surface for which the sets of essential and log essential valuations do not coincide.
Let $X$ be the hypersurface in $\C^3$ defined by the equation $f=z^2+(x^3+y^3)(y^3+x^4)$. 
Consider the projection $X\to {\mathbb C}^2$ defined by the coordinates $x$ and $y$: it is a double cover branched on the curve $C\coloneqq \{(x^3+y^3)(y^3+x^4)=0\}$.
Blow up the origin in ${\mathbb C}^2$, and let $Y$ be the surface obtained by base change and normalization: $Y$ is a double cover branched on the strict transform $C'$ of $C$ and it is smooth since $C'$ is smooth. 
The exceptional locus of the resolution $Y\to X$ is the inverse image $D\subset Y$ of the exceptional curve of the blow up. 
A standard computation shows that $D$ is irreducible and $D^2=-2$, so the resolution is minimal.
Moreover, $D$ has a simple cusp as singularity, therefore $(Y,D)$ is not a log resolution of $(X,X_{sing})$.
\end{ex}
%
%
%



\begin{rem}\label{remark_existence_resolutions}
We expect the approach used in this section to lead to a new proof of the existence of resolutions of surfaces, at least in characteristic 0, in a similar way as one can prove the semistable reduction theorem for curves using non-archimedean analytic spaces.
A proof would go roughly as follows.
The normalized space $\Txz$ can be covered by finitely many smooth affinoid $k((t))$-analytic curves, since all the points of $\Xeta$ are regular.
Then \cite[5.1.14]{Duc} applied to those $k((t))$-analytic curves gives us a vertex set $S\subset\Div_{X,Z}$ such that all connected components of $\Txz\setminus S$ are \defi{virtual discs} or \defi{virtual annuli}, i.e. $k((t))$-analytic spaces that become a $k((t))$-disc or a $k((t))$-annulus after a finite separable extension of $k((t))$.
If we could prove that all those virtual discs and annuli are actual discs and annuli, we would obtain a log resolution of $(X,Z)$, since by enlarging $S$ we can cut an annulus of modulus $n$ into $n$ standard annuli.
If the characteristic of $k$ is zero, by a special case of \cite{Ducros13} every virtual disc is a disc. 
A virtual annulus is a pseudo-affinoid $k((t))$-analytic space, and to prove that it is an annulus it would be enough, by a slight generalization of Proposition~\ref{prop_discs}, to show that it is a distinguished pseudo-affinoid.
We believe that it is always possible to enlarge $S$ further and break a given virtual annulus in discs and finitely many annuli.
\end{rem}

\bibliographystyle{alpha}
\bibliography{../biblioarboreti}

\end{document}